\documentclass[preprint]{elsarticle}

 \usepackage{amsmath,amssymb,amsthm,amsfonts}
 \usepackage{epstopdf}
 \usepackage{arydshln,geometry}
\usepackage{pgfpages,tikz,tikz-cd,pgfkeys,pgfplots}
\usetikzlibrary{arrows,positioning,matrix,fit,backgrounds,shapes}
 \usepackage[dvips]{epsfig}
 \usepackage{color}
 \definecolor{LemonChiffon}{rgb}{100, 98, 80}
\definecolor{myblue}{rgb}{0,0.4,0.8}
\definecolor{orange}{rgb}{1, 0.4, 0}
\definecolor{mygreen}{rgb}{0, 0.8, 0.2}
\definecolor{myred}{rgb}{204, 0, 0}
\definecolor{violet}{RGB}{0.4,0.2,1}
\definecolor{brown}{rgb}{0.6, 0.4, 0}

\usepackage{graphicx}
\usepackage{xcolor}
\usepackage{pgf}
\usepackage{soul}
\newcounter{statement}
\newcommand{\statement}[2]{%
	\begin{equation}\refstepcounter{statement}\tag{S\thestatement}\label{stm:#1}%
		\parbox{\dimexpr\linewidth-4em}{#2}%
	\end{equation}%
}

\theoremstyle{plain}
\newtheorem{theorem}{Theorem}[section]
\newtheorem{lemma}[theorem]{Lemma}
\newtheorem{prop}[theorem]{Proposition}

\newtheorem{prb}[theorem]{Problem}

\newtheorem{question}[theorem]{Question}

\theoremstyle{definition}
\newtheorem{defi}[theorem]{Definition}

\newtheorem{exm}[theorem]{Example}

\usetikzlibrary{arrows,positioning,matrix,fit,backgrounds,shapes,shapes.geometric, intersections}
\usetikzlibrary{shapes.callouts,decorations.pathmorphing, calc}
\usetikzlibrary{decorations.markings}
\tikzstyle{vertex}=[circle, draw, inner sep=0pt, minimum size=3pt] 
\newcommand{\vertex}{\node[vertex]}

\begin{document}

\begin{frontmatter}

 \title{On Toeplitz graphs being line graphs}

\author[label2]{Gi-Sang Cheon}
\ead{gscheon@skku.edu}
\author[label4]{Bumtle Kang\corref{cor1}}
\ead{thelokbt@ksa.kaist.ac.kr}
\cortext[cor1]{Corresponding author}
\author[label2,label3]{Suh-Ryung Kim}
\ead{srkim@snu.ac.kr}
\author[label1]{Seyed Ahmad Mojallal}
\ead{ahmad\_mojalal@yahoo.com}
\author[label3]{Homoon Ryu}
\ead{ryuhomun@naver.com}

\address[label1]{Department of Mathematics and Statistics, University of Regina, Regina, Saskatchewan, S4S0A2, Canada}
\address[label2]{Applied Algebra and Optimization Research Center, Department of Mathematics, Sungkyunkwan University, Suwon $16419$, Republic of Korea}
\address[label3]{Department of Mathematics Education, Seoul National University, Seoul $08826$, Republic of Korea}
\address[label4]{Department of Mathematics and Computer Science, Korea Science Academy of KAIST, Busan $47162$, Republic of Korea}

\begin{abstract}

 A Toeplitz graph $T_n \langle t_1,t_2,\ldots,t_k\rangle$ is a simple graph with the vertex set $[n]$ such that two vertices $v$ and $w$ are adjacent if and only if $|v-w| = t_i$ for some $i \in [k]$.
    In this paper, we investigate line Toeplitz graphs, which are Toeplitz graphs that happen to be line graphs.
  We first show that for a sufficiently large $n$, the family of claw-free Toeplitz graphs of order $n$  is $T_n \langle t,2t,\ldots,kt\rangle$ for some nonnegative integers $t$ and $k$.
  Interestingly, this family consists of a union of Toeplitz graphs each of which is isomorphic to a $k$-tree the notion of which was introduced by Patil in 1986.
    Then we completely characterize $T_n \langle t,2t,\ldots,kt\rangle$ for any positive integer $n$ that is a line graph.
  Furthermore, we provide a comprehensive description of a line Toeplitz graph $T_n \langle t_1,t_2\rangle$ and $T_n \langle t_1,t_2,t_3\rangle$.
  In general, line Toeplitz graph seems very challenging to characterize completely. Even for $T_n \langle t_1,t_2,t_3\rangle$, it was not easy to do so.
  It is also worth mentioning that there is a line Toeplitz graph that is  not in the form $T_n \langle t,2t,3t\rangle$.
\end{abstract}


\end{frontmatter}

 \section{Introduction}

 A {\it Toeplitz graph} is a simple graph whose adjacency matrix exhibits the Toeplitz property, meaning that the $(0,1)$-pattern in the matrix remains the same along both rows and columns.
According to the characteristics of the Toeplitz property, we can denote a Toeplitz graph with $n$ vertices by $T_n\langle t_1,\ldots, t_k\rangle$
 if the $1$'s in the first row of its adjacency matrix are placed at positions $1+t_1$, \ldots, $1+t_k$ with $1\le t_1<\ldots <t_k<n$.
  In addition, we label the vertices of the Toeplitz graph with $1,\ldots,n$ so that the $i$th row of its adjacency matrix corresponds to the vertex labeled $i$.

  Toeplitz graphs have been first investigated with respect to hamiltonicity by van Dal {\it et al.} in 1996.
 \cite{DTT} (see also \cite{Heu,MQ,MZ}). Since then some properties of Toeplitz graphs have been studied by several authors. Indeed, bipartite Toeplitz
 graphs have been fully characterized in terms of bases and circuits by Euler {\it et al.}~\cite{Eul,ELZ}. Colouring aspects are especially treated in \cite{Heu2,KM,NP}.
 Planar Toeplitz graphs have been also fully characterized by Euler \cite{Eul3,EZ} providing, in particular, a complete description of the class of 3-colourable such
 graphs.
 Toeplitz graphs have also various applications in different fields. For instance, Toeplitz graphs can be used to model the structure of the data and extract features in time series or image data in signal processing. This helps in understanding and analyzing the periodicity or localized patterns in the signals. For example, see \cite{CGSXZ,CHL}

 It is also known \cite{EvL1,EvL2} that the line graph in graph theory is an important tool in network analysis. Creating the line graph of a given graph connects the edges of the original graph as vertices in the new graph. This allows for investigating and analyzing the network structure of the original graph.
Formally, the {\it line graph} $G$ of a graph $H$ has the edge set of $H$ as its vertex set and two vertices are adjacent in $G$ if and only if their corresponding edges are adjacent in $H$.

In this paper, we research the Toeplitz graph, which is a line graph, in the sense of combining two graphs that have various applications and as  a follow-up to Mojallal {\it et al.}~\cite{CJKKM}.

\section{Preliminaries}

 A {\it chordal graph} is a simple graph with no chordless cycle of length at least $4$ as an induced subgraph.
  A graph is an {\it interval graph} if it captures the intersection relation for some set of intervals on the real line.
  A {\it clique} is a complete subgraph or a subset of vertices of an undirected graph such that every two distinct vertices in the clique are adjacent. The {\it clique number} $\omega(G)$ of a graph $G$ is the maximum number of vertices among the cliques of $G$.
Throughout the paper, we denote the $n$-set $\{1,\ldots,n\}$ by $[n]$.

Very recently, Mojallal {\it et al.}~\cite{CJKKM} studied structures of Toeplitz graphs and obtained the following interesting results.
 \begin{lemma}[\cite{CJKKM}]\label{chordal}
 	Let $G=T_n\langle t_1, \ldots, t_k\rangle$. If
 	$n \ge t_{k-1}+t_k$, then the following statements are equivalent.
 	\begin{itemize}
 		\item[{\rm (i)}] $G$ is interval;
 		\item[{\rm (ii)}] $G$ is chordal;
 		\item[{\rm (iii)}] $t_i=it_1$ for all $i\in [k]$;
 		\item[{\rm (iv)}] $\omega(G)=k+1$.
 	\end{itemize}
 \end{lemma}

 Furthermore, a Toeplitz graph satisfying the condition (iii) of  Lemma~\ref{chordal} has a noteworthy structure described as follows.
\begin{prop}[\cite{CJKKM}]\label{lem1} For positive integers $k,t$, let $G=T_n\langle t, 2t,\ldots, kt\rangle$.
	Then $G$ has $t$ components. In particular, if $H_1, \ldots, H_t$ are the components of $G$, then $H_i$ is isomorphic to the
	graph $T_{\lfloor(n-i)/t\rfloor+1}\langle 1, 2,\ldots, k\rangle$ and the vertex set of $H_i$
	is $\{i+st\;|\; s=0,1,\ldots, \lfloor (n-i)/t\rfloor\}$ for each $i\in
	[t]$.
\end{prop}
 This proposition tells us that each component of $T_n\langle t, 2t,\ldots, kt\rangle$ is a lump of consecutive cliques of size $k+1$ sharing $k$ vertices.
 As a matter of fact, this graph can be described in terms of $k$-trees whose notion was introduced by Patil.

\begin{defi}[\cite{PH1}]\label{def:k} The family of $k$-trees (or $K_k$-trees) is the set of all graphs that can be obtained by the following recursive construction procedure:
  \begin{itemize}
  \item[{\rm 1.}] The complete graph $K_k$ of order $k$ is the smallest $k$-tree.
  \item[{\rm 2.}] To a $k$-tree $G$ with $n-1$ vertices for $n\ge k+1$, add a new vertex and make it adjacent to any $k$ mutually adjacent vertices of $G$, so that the resulting $k$-tree is of order $n$.
  \end{itemize}
    \end{defi}
 Note that a component of $T_n\langle t, 2t,\ldots, kt\rangle$ is $H_i=T_{\lfloor(n-i)/t\rfloor+1}\langle 1, 2,\ldots, k\rangle$ given in Proposition~\ref{lem1} for some $i \in [t]$.
 It can easily be checked that $H_i$ is obtained from $K_{k+1}$ by adding the vertex $(k+1)+j$ so that $\{1+j,\ldots,(k+1)+j\}$ forms a clique at the $j$th step for each $1\le j \le \lfloor (n-i)/t\rfloor-k$ and thus, by definition, $H_i$ is a $(k+1)$-tree.
 Therefore $T_n\langle t, 2t,\ldots, kt\rangle$ is a union of $(k+1)$-trees.

For example, we take $G:=T_{30} \langle 5,10,15 \rangle$.
As shown in Figure~\ref{fig:Toeplitz1}, $G$ consists of five $4$-trees of order $6$.
Moreover, the adjacency matrix of each tree is simultaneous permutation equivalent to the $6 \times 6$ Toeplitz matrix
 \[
 \begin{tikzpicture}[baseline={([yshift=-1ex]current bounding box.center)}]
 \matrix[matrix of math nodes,left delimiter=(,right delimiter=).] (M)
 {  0 & 1 & 1 & 1 &  0 & 0  \\
 1 & 0 & 1 & 1 &  1 & 0   \\
 1 & 1 & 0 & 1 &  1 & 1 \\
 1 & 1 & 1 & 0 &  1 & 1 \\
 0 & 1 & 1 & 1 &  0 & 1 \\
 0 & 0 & 1 & 1 &  1 & 0 \\
 };
 \node[draw,fit=(M-1-1) (M-4-4),inner sep=0pt]{};
 \node[draw,fit=(M-2-2) (M-5-5),inner sep=0pt]{};
 \node[draw,fit=(M-3-3) (M-6-6),inner sep=0pt]{};
  \end{tikzpicture}
 \]

 We observe that the diagonal blocks are $J_4 - I_4$ and two consecutive blocks overlap up to $J_3-I_3$. In addition, the adjacency matrix of $G$ is the direct sum of five of this $6 \times 6$ Toeplitz matrix.
 We generalize this idea of describing the adjacency matrix of $G$ into the  adjacency matrix of  $T_n\langle t, 2t,\ldots, kt\rangle$. We define a matrix, denoted by $A_i(n,k,t)$, for each $i \in [t]$ as follows:\begin{itemize}  \item[(1)] the order of  $A_i(n,k,t)$ is $\lfloor(n-i)/t\rfloor+1$;
  \item[(2)] the diagonal blocks of $A_i(n,k,t)$ are $J_{k+1} - I_{k+1}$, two consecutive blocks overlap up to $J_{k}-I_{k}$, and the remaining entries are all zeros.
\end{itemize}
By this definition and Lemma~\ref{lem1}, the adjacency matrix of $T_n\langle t, 2t,\ldots, kt\rangle$ is similar to the following matrix:\[
A_1(n,k,t) \oplus A_2(n,k,t) \oplus \cdots \oplus A_t(n,k,t)\]

\begin{figure}
	\begin{center}
		\begin{tikzpicture}
			[scale=0.8,every node/.style={scale=0.8}]
			
			\filldraw (-5,0) circle (2pt);
			\draw (-5,-0.3) node{\small $1$};
			\filldraw (-3,0) circle (2pt);
			\draw (-3,-0.3) node{\small $6$};
			\filldraw (-1,0) circle (2pt);
			\draw (-1,-0.3) node{\small $21$};
			
			\filldraw (-4,2) circle (2pt);
			\draw (-4,2.3) node{\small $11$};
			\filldraw (-2,2) circle (2pt);
			\draw (-2,2.3) node{\small $16$};
			\filldraw (0,1) circle (2pt);
			\draw (0,1.3) node{\small $26$};
			
			\draw (-5,0) -- (-3,0);
			\draw (-5,0) -- (-4,2);
			\draw (-5,0) -- (-2,2);
			\draw (-3,0) -- (-4,2);
			\draw (-3,0) -- (-2,2);
			\draw (-3,0) -- (-1,0);
			\draw (-4,2) -- (-2,2);
			\draw (-4,2) -- (-1,0);
			\draw (-4,2) -- (0,1);
			\draw (-2,2) -- (-1,0);
			\draw (-2,2) -- (0,1);
			\draw (-1,0) -- (0,1);
			
			\filldraw (1,0) circle (2pt);
			\draw (1,-0.3) node{\small $2$};
			\filldraw (3,0) circle (2pt);
			\draw (3,-0.3) node{\small $7$};
			\filldraw (5,0) circle (2pt);
			\draw (5,-0.3) node{\small $22$};
			
			\filldraw (2,2) circle (2pt);
			\draw (2,2.3) node{\small $12$};
			\filldraw (4,2) circle (2pt);
			\draw (4,2.3) node{\small $17$};
			\filldraw (6,1) circle (2pt);
			\draw (6,1.3) node{\small $27$};
			
			\draw (1,0) -- (3,0);
			\draw (1,0) -- (2,2);
			\draw (1,0) -- (4,2);
			\draw (3,0) -- (2,2);
			\draw (3,0) -- (4,2);
			\draw (3,0) -- (5,0);
			\draw (2,2) -- (4,2);
			\draw (2,2) -- (5,0);
			\draw (2,2) -- (6,1);
			\draw (4,2) -- (5,0);
			\draw (4,2) -- (6,1);
			\draw (5,0) -- (6,1);

			\filldraw (-2,-5.5) circle (2pt);
			\draw (-2,-5.8) node{\small $4$};
			\filldraw (0,-5.5) circle (2pt);
			\draw (0,-5.8) node{\small $9$};
			\filldraw (2,-5.5) circle (2pt);
			\draw (2,-5.8) node{\small $24$};
			
			\filldraw (-1,-3.5) circle (2pt);
			\draw (-1,-3.2) node{\small $14$};
			\filldraw (1,-3.5) circle (2pt);
			\draw (1,-3.2) node{\small $19$};
			\filldraw (3,-4.5) circle (2pt);
			\draw (3,-4.2) node{\small $29$};
			
			\draw (-2,-5.5) -- (0,-5.5);
			\draw (-2,-5.5) -- (-1,-3.5);
			\draw (-2,-5.5) -- (1,-3.5);
			\draw (0,-5.5) -- (-1,-3.5);
			\draw (0,-5.5) -- (1,-3.5);
			\draw (0,-5.5) -- (2,-5.5);
			\draw (-1,-3.5) -- (1,-3.5);
			\draw (-1,-3.5) -- (2,-5.5);
			\draw (-1,-3.5) -- (3,-4.5);
			\draw (1,-3.5) -- (2,-5.5);
			\draw (1,-3.5) -- (3,-4.5);
			\draw (2,-5.5) -- (3,-4.5);
			
			\filldraw (3,-3) circle (2pt);
			\draw (3,-3.3) node{\small $3$};
			\filldraw (5,-3) circle (2pt);
			\draw (5,-3.3) node{\small $8$};
			\filldraw (7,-3) circle (2pt);
			\draw (7,-3.3) node{\small $23$};
			
			\filldraw (4,-1) circle (2pt);
			\draw (4,-0.7) node{\small $13$};
			\filldraw (6,-1) circle (2pt);
			\draw (6,-0.7) node{\small $18$};
			\filldraw (8,-2) circle (2pt);
			\draw (8,-1.7) node{\small $28$};
			
			\draw (3,-3) -- (5,-3);
			\draw (3,-3) -- (4,-1);
			\draw (3,-3) -- (6,-1);
			\draw (5,-3) -- (4,-1);
			\draw (5,-3) -- (6,-1);
			\draw (5,-3) -- (7,-3);
			\draw (4,-1) -- (6,-1);
			\draw (4,-1) -- (7,-3);
			\draw (4,-1) -- (8,-2);
			\draw (6,-1) -- (7,-3);
			\draw (6,-1) -- (8,-2);
			\draw (7,-3) -- (8,-2);
			
			\filldraw (-7,-3) circle (2pt);
			\draw (-7,-3.3) node{\small $5$};
			\filldraw (-5,-3) circle (2pt);
			\draw (-5,-3.3) node{\small $10$};
			\filldraw (-3,-3) circle (2pt);
			\draw (-3,-3.3) node{\small $25$};
			
			\filldraw (-6,-1) circle (2pt);
			\draw (-6,-0.7) node{\small $15$};
			\filldraw (-4,-1) circle (2pt);
			\draw (-4,-0.7) node{\small $20$};
			\filldraw (-2,-2) circle (2pt);
			\draw (-2,-1.7) node{\small $30$};
			
			\draw (-7,-3) -- (-5,-3);
			\draw (-7,-3) -- (-6,-1);
			\draw (-7,-3) -- (-4,-1);
			\draw (-5,-3) -- (-6,-1);
			\draw (-5,-3) -- (-4,-1);
			\draw (-5,-3) -- (-3,-3);
			\draw (-6,-1) -- (-4,-1);
			\draw (-6,-1) -- (-3,-3);
			\draw (-6,-1) -- (-2,-2);
			\draw (-4,-1) -- (-3,-3);
			\draw (-4,-1) -- (-2,-2);
			\draw (-3,-3) -- (-2,-2);
		\end{tikzpicture}
	\end{center}
	\caption{The Toeplitz graph $T_{30}\langle5,10,15\rangle$ consisting of five $4$-trees of order $6$}\label{fig:Toeplitz1}
	
\end{figure}

A {\it claw} of a graph means the complete bipartite graph $K_{1,3}$ as an induced subgraph.
In this paper, we denote by  $(a;b,c,d)$ a $K_{1,3}$ as an induced subgraph with center $a$ and leaves $b,c,d$.
Let $G$ be a Toeplitz graph $T_n\langle t_1, \ldots, t_k \rangle$ and $\mathcal{T} = \{t_1,\ldots,t_k\}$.
We note that $a$, $b$, $c$, and $d$ form the claw $(a;b,c,d)$ in $G$ if and only if
\begin{itemize}
\item[(i)] $|a-j| \in \mathcal{T}$ for all $j \in \{b,c,d\}$;
\item[(ii)] $|b-c| \notin \mathcal{T}$;
\item[(iii)] $|b-d| \notin \mathcal{T}$;
\item[(iv)] $|c-d| \notin \mathcal{T}$.
\end{itemize}

If  $a$, $b$, $c$, and $d$ satisfy the conditions (i) (resp. (i)-(ii), (i)-(iii)), then we say that  $a$, $b$, $c$, and $d$ {\it form $[ a;b,c,d ]$ (resp. $[ a;b\nsim c, d ]$, $[ a;b\nsim c \nsim d ]$)}. See Figure~\ref{fig:claws} for an illustration.

By the definition of Toeplitz graph, if there is a $K_{1,3}$ (resp.\ claw) $(a;b,c,d)$ in $T_n\langle t_1, \ldots, t_k \rangle$, then $((n+1)-a; (n+1)-b, (n+1)-c, (n+1)-d)$ is also a $K_{1,3}$ (resp.\ claw).
In this vein, we may consider only $K_{1,3}$ satisfying the following:
\begin{itemize}
\item[($\star$)] at least two of $b,c,d$ are greater than $a$ for $(a;b,c,d)$.
\end{itemize}

The {\it line graph} $G$ of a graph $H$ has the edge set of $H$ as its vertex set and two vertices are adjacent in $G$ if and only if their corresponding edges are adjacent in $H$.

This paper aims to characterize the line graph $T_n\langle t_1,t_2\rangle$ and $T_n\langle t_1,t_2,t_3 \rangle$. (Theorems~\ref{thm:linegraph}, \ref{prop:line2}, and \ref{thm:line3})
It is natural to characterize $T_n\langle t_1,t_2 \rangle$  and $T_n\langle t_1,t_2,t_3 \rangle$ (Theorems~\ref{pro2}, \ref{thm:equiv}, \ref{thm0}, and \ref{thm:cond3}) as a precursor due to the fact that a line graph is claw-free.
By doing so, we add one more equivalent statement to the existing result Lemma~\ref{chordal}.

\begin{figure}
\begin{center}
\begin{tikzpicture}[scale=0.6,x=0.75pt,y=0.75pt,yscale=-1,xscale=1]
\filldraw (173,129) circle (2pt);
\filldraw (251,228) circle (2pt);
\filldraw (100,228) circle (2pt);
\filldraw (173,270) circle (2pt);
\draw    (173,129) -- (251,228) ;
\draw    (173,129) -- (100,228) ;
\draw    (173,129) -- (173,270) ;
\draw[densely dashed]   (100,228) -- (251,228) ;
\draw[densely dashed]    (173,270) -- (100,228) ;
\draw[densely dashed]    (251,228) -- (173,270) ;

\draw (166,107) node [anchor=north west][inner sep=0.75pt]   [align=left] {$a$};
\draw (92,235) node [anchor=north west][inner sep=0.75pt]   [align=left] {$b$};
\draw (167,283) node [anchor=north west][inner sep=0.75pt]   [align=left] {$c$};
\draw (246,234) node [anchor=north west][inner sep=0.75pt]   [align=left] {$d$};
\draw (132,305) node [anchor=north west][inner sep=0.75pt]   [align=left] {$(a;b,c,d)$};
\end{tikzpicture}\quad
\begin{tikzpicture}[scale=0.6,x=0.75pt,y=0.75pt,yscale=-1,xscale=1]
\filldraw (173,129) circle (2pt);
\filldraw (251,228) circle (2pt);
\filldraw (100,228) circle (2pt);
\filldraw (173,270) circle (2pt);
\draw    (173,129) -- (251,228) ;
\draw    (173,129) -- (100,228) ;
\draw    (173,129) -- (173,270) ;

\draw (166,107) node [anchor=north west][inner sep=0.75pt]   [align=left] {$a$};
\draw (92,235) node [anchor=north west][inner sep=0.75pt]   [align=left] {$b$};
\draw (167,283) node [anchor=north west][inner sep=0.75pt]   [align=left] {$c$};
\draw (246,234) node [anchor=north west][inner sep=0.75pt]   [align=left] {$d$};
\draw (132,305) node [anchor=north west][inner sep=0.75pt]   [align=left] {$[a;b, c,d]$};
\end{tikzpicture}\quad
\begin{tikzpicture}[scale=0.6,x=0.75pt,y=0.75pt,yscale=-1,xscale=1]
\filldraw (173,129) circle (2pt);
\filldraw (251,228) circle (2pt);
\filldraw (100,228) circle (2pt);
\filldraw (173,270) circle (2pt);
\draw    (173,129) -- (251,228) ;
\draw    (173,129) -- (100,228) ;
\draw    (173,129) -- (173,270) ;
\draw[densely dashed]    (173,270) -- (100,228) ;

\draw (166,107) node [anchor=north west][inner sep=0.75pt]   [align=left] {$a$};
\draw (92,235) node [anchor=north west][inner sep=0.75pt]   [align=left] {$b$};
\draw (167,283) node [anchor=north west][inner sep=0.75pt]   [align=left] {$c$};
\draw (246,234) node [anchor=north west][inner sep=0.75pt]   [align=left] {$d$};
\draw (127,305) node [anchor=north west][inner sep=0.75pt]   [align=left] {$[a;b\nsim c,d]$};
\end{tikzpicture}\quad
\begin{tikzpicture}[scale=0.6,x=0.75pt,y=0.75pt,yscale=-1,xscale=1]
\filldraw (173,129) circle (2pt);
\filldraw (251,228) circle (2pt);
\filldraw (100,228) circle (2pt);
\filldraw (173,270) circle (2pt);
\draw    (173,129) -- (251,228) ;
\draw    (173,129) -- (100,228) ;
\draw    (173,129) -- (173,270) ;
\draw[densely dashed]    (173,270) -- (100,228) ;
\draw[densely dashed]    (251,228) -- (173,270) ;

\draw (166,107) node [anchor=north west][inner sep=0.75pt]   [align=left] {$a$};
\draw (92,235) node [anchor=north west][inner sep=0.75pt]   [align=left] {$b$};
\draw (167,283) node [anchor=north west][inner sep=0.75pt]   [align=left] {$c$};
\draw (246,234) node [anchor=north west][inner sep=0.75pt]   [align=left] {$d$};
\draw (127,305) node [anchor=north west][inner sep=0.75pt]   [align=left] {$[a;b\nsim c\nsim d]$};
\end{tikzpicture}
\end{center}
\caption{The configurations corresponding to $(a;b,c,d)$, $[a;b,c,d]$, $[a;b\nsim c,d]$, and $[a;b \nsim c \nsim d]$, respectivly. Solid edges represents edges that must be present, and dotted edges indicate that their two endpoints are not adjacent.}\label{fig:claws}
\end{figure}

\section{Equivalence conditions for a Toeplitz graph being claw-free}\label{sec2}

   The section begins with the following question:
 \begin{question}\label{que1}
 Which Toeplitz graphs $T_n\langle t_1, \ldots, t_k\rangle$ are claw-free?
\end{question}

Let $G$ be a Toeplitz graph with $n$ vertices. If $G=T_n\langle
t_1\rangle$ then $G$ is the union of paths by Proposition~\ref{lem1}, which is
 claw-free.
Accordingly, we only consider Toeplitz graphs $T_n\langle t_1,\ldots, t_k\rangle$ for $k \ge 2$.

 \begin{theorem}\label{thm:main1}
  A Toeplitz graph $T_n\langle t,2t, \ldots, kt\rangle$ is claw-free.
 \end{theorem}

\begin{proof}
To reach a contradiction, suppose that $T_n\langle t,2t, \ldots, kt\rangle$ contains a claw $(j;i_1,i_2,i_3)$.
By $(\star)$, we may assume that $i_2>i_1>j$.  Then $i_2=j+rt$ and $i_1=j+st$ for some  $1\le s<r \le k$, so $i_2-i_1=(r-s)t$.
Since $1 \le r-s \le k-1$, $i_1$ and $i_2$ are adjacent and we reach a contradiction.
Thus $T_n\langle t,2t, \ldots, kt\rangle$ is claw-free.
\end{proof}
 The converse of Theorem~\ref{thm:main1} is true if the number of vertices is sufficiently large as shown below.

\begin{lemma}\label{lem:dkd}
	Let $B$ be a nonempty subset of positive integers which is  bounded above and is closed under difference of distinct elements.
	Then $B = \{ d, 2d, \ldots, kd\}$ for some positive integers $d$ and $k$.
\end{lemma}
\begin{proof}
	By the Well-Ordering Principle, there is a minimum element $d$ in $B$.
	Suppose there is an element $a$ distinct from $d$ in $B$.
	Then $a = qd + r$ for some integers $q$ and $r$, $0 \le r < d$.
	Now $a-d \in B$ by the hypothesis.
	Then $(a-d)-d \in B$.
	We continue in this way to have $a-qd = r \in B$ unless $r = 0$.
	Thus, by the minimality of $d$, $r = 0$.
	Thus every element in $B$ is a multiple of $d$.
	Since $B$ is bounded above, there is a maximum element $b$ in $B$.
	Then $b = kd$ for some positive integer $k$.
	Now $ b- d = (k-1)d \in B$ by the hypothesis.
	Then $(b-d)-d = (k-2)d \in B$.
	We continue in this way to have $B = \{d,2d, \ldots, kd\}$.
\end{proof}

\begin{theorem}\label{thm:tktl}
	Suppose that $T_n \langle t_1, \ldots, t_k\rangle$ is claw-free with $n > t_k+t_\ell$ for some $1 \le \ell < k$.
	Then $t_i = it_1$ for any $1 \le i \le \ell$. Furthermore, if $\ell=k-1$, then $t_k = kt_1$.
\end{theorem}
\begin{proof}
	Let $G=T_n \langle t_1, \ldots, t_k\rangle$.
	Take two distinct positive integers $i, j \le \ell$.
Then we can check that  $1$, $1+t_k$, $1+t_k+t_i$, and $1+t_k+t_j$ form $[1+t_k;1+t_k+t_i\nsim 1 \nsim 1+t_k+t_j]$.
	 Since $G$ is claw-free, $1+t_k+t_i$ and $1+t_k+t_j$ are adjacent.
	Therefore $|t_i-t_j| \in \{t_1, \ldots, t_k\}$.
	Since $|t_i-t_j| < t_\ell$, $|t_i-t_j| \in \{t_1, \ldots, t_\ell\}$.
	Since we have chosen $i, j$ arbitrarily as long as $i, j \le \ell$, $\{t_1, \ldots, t_\ell \}$ is closed under difference of two distinct elements.
	Thus, by Lemma~\ref{lem:dkd}, $t_i = it_1$ for any $1 \le i \le \ell$.

To show the `furthermore' part, suppose that $\ell=k-1$.
We note that $1$, $1+t_k$,  $1+t_k-t_1$, and $1+t_{k-1}+t_k$ form
$[1+t_k; 1+t_{k-1}+t_k \nsim 1, 1+t_k-t_1]$.
	Since $G$ is claw-free, one of $(1+t_k-t_1)-1 = t_k-t_1$ and $(1+t_{k-1}+t_k)-(1+t_k-t_1)=t_1+t_{k-1}$ belongs to $\{t_1,\ldots,t_k\}$.
	If $t_k-t_1 \in \{t_1,\ldots,t_k\}$, then $t_k-t_1 = t_{k-1}$ since $t_i = it_1$ for each $i \in \{1,\ldots, k-1\}$.
	If $t_1+t_{k-1} \in \{t_1,\ldots,t_k\}$, then $t_k= t_1+t_{k-1}$ since $t_1+t_{k-1} > t_{k-1}$.
	 Therefore $t_k = kt_1$ in both cases and it completes the proof.
\end{proof}

%

 Theorems~\ref{thm:main1} and \ref{thm:tktl}  generalize Theorem~\ref{chordal} as follows:

\begin{theorem}\label{pro2}
	Let $G=T_n\langle t_1,\ldots, t_k\rangle$. If
	$n > t_{k-1}+t_k$, then the following statements are equivalent.
	\begin{itemize}
		\item[{\rm (i)}] $G$ is claw-free;
		\item[{\rm (ii)}] $G$ is interval;
		\item[{\rm (iii)}] $G$ is chordal;
		\item[{\rm (iv)}] $t_i=it_1$ for each $i \in [k]$;
		\item[{\rm (v)}] $\omega(G)=k+1$.
	\end{itemize}
\end{theorem}

Theorem~\ref{pro2} gives a meaningful characterization for a claw-free Toeplitz graphs.
However, these descriptions are no longer equivalent without constraints on n.
The following theorem shows that the lower bound on $n$ is tight.

\begin{lemma} A Toeplitz graph $T_n\langle t,2t, \ldots,(k-1)t,(k+1)t\rangle$ with $n \le 2kt$  is claw-free.\label{lem:mutation}
\end{lemma}
\begin{proof} Let $G=T_n\langle t,2t, \ldots,(k-1)t,(k+1)t\rangle$ and $B = \{t,2t,\ldots,(k-1)t,(k+1)t\}$.
To reach a contradiction, suppose that $G$ has a claw $(u;x,y,z)$. By $(\star)$, we may assume $u<y<z$.
Then $\{y-u,z-u\} \subset B$. Since $z-y = (z-u)-(y-u)$, $z-y \notin B$ only if $z-u = (k+1)t$ and $y-u = t$.
If $u<x$, then $0< x-u < y-u$ and $x-u \in B$, which is impossible.
Suppose that $u>x$. Then $u-x \in B$.
If $u-x \ge (k-1)t$, then $z-x = (z-u)+(u-x) \ge 2kt$, which is impossible since $n \le 2kt$.
Thus $u-x < (k-1)t$. Since $y-x = (y-u)+(u-x)$, $y-x$ is a multiple of $t$ in $\{t,\ldots,(k-1)t\}$ and we reach a contradiction.
Hence $G$ is claw-free.
\end{proof}

 \begin{theorem} \label{thm:equiv}
 Let $G=T_n\langle t_1,\ldots, t_k\rangle$ be a Toeplitz graph with $k \ge 4$ and $t_{k-1}+t_k = n$. Then $G$ is claw-free if and only if $t_i=it_1$ for each $i \in [k-1]$ and $t_k \in \{kt_1,(k+1)t_1\}$. Furthermore, if $n$ is odd, then $G$ is claw-free if and only if $t_i = it_1$ for each $i \in [k]$.
 \end{theorem}

\begin{proof} The `if' part of the main text and the `furthermore' part
immediately follow from Theorem~\ref{thm:main1} and Lemma~\ref{lem:mutation}.

We show the `only if' part. Since $G:=T_n\langle t_1,\ldots, t_k\rangle$ is claw-free and $t_{k-2}+t_k < t_{k-1} +t_k =n$, \begin{equation}
t_j = jt_1 \ \ \mbox{for each} \ \ 1\le j \le k-2 \label{eq:k2} \end{equation} by Theorem~\ref{thm:tktl}.
For convenience, let $B=\{t_1,\ldots,t_k\}$.

We first claim that $t_{k-1} = (k-1)t_1$.
We  note that $1$, $1+t_{k-2}$, $1+t_{k-2}+t_1$, and $1+t_{k-2}+t_{k-1}$ form
$[1+t_{k-2};1,1+t_{k-2}+t_1,1+t_{k-2}+t_{k-1}]$.
Since $G$ is claw-free, one of $t_{k-2}+t_1$, $t_{k-1}-t_1$, or $t_{k-2}+t_{k-1}$ belongs to $B$.
If $t_{k-1}-t_1 \in B$, then $t_{k-1}-t_1 = t_{k-2}$ and so $t_{k-1} = (k-1)t_1$ by \eqref{eq:k2}.
Consider the case $t_{k-2}+t_1 \in B$. Then $t_{k-2}+t_1$ equals either $t_{k-1}$ or $t_k$.
If $t_{k-2}+t_1 = t_k$, then $t_{k-1}-t_1$ and $t_{k-1}+t_1$ do not belong to $B$, so $(1+t_{k-1};1, 1+t_{k-1}+t_1,1+2t_{k-1})$ is a claw and we reach a contradiction.
Therefore $t_{k-2}+t_1 = t_{k-1}$ and so $t_{k-1} = (k-1)t_1$ by \eqref{eq:k2}.
Now we consider the case $t_{k-2}+t_{k-1} \in B$. Then $t_{k-2} +t_{k-1} = t_k$.
We observe that $1$, $1+t_{k-1}$, $1+2t_{k-1}$, and $1+t_{k-1}-t_{1}$ form
$[1+t_{k-1};1\nsim 1+2t_{k-1} \nsim 1+t_{k-1}-t_{1}]$.
Since $G$ is claw-free, $t_{k-1}-t_1 \in B$ and so $t_{k-1} = (k-1)t_1$ by \eqref{eq:k2}.

Finally we show that $t_k=kt_1$ or $t_k=(k+1)t_1$. We note that $1$, $1+t_k$, $1+ t_k - t_2$, and $1+t_k+t_{k-2}$ form
$[1+t_k;1 \nsim 1+t_k+t_{k-2}, 1+t_k-t_2]$.
Since $G$ is claw-free, one of $(1+t_k-t_2)-1 =t_k-t_2$ and $(1+t_k+t_{k-2}) - (1+t_k-t_2)=t_{k-2}+t_2$ belongs to $B$.
If $t_{k-2}+t_2 \in B$, then $t_k = kt_1$ since $t_j= jt_1$ for each $1 \le j \le k-1$.
Consider the case $t_k-t_2 \in B$. Since $t_j = jt_1$ for each $1 \le j \le k-1$, $t_k-t_2 = t_{k-2}$ implying $t_k=k t_1$ or $t_k-t_2 = t_{k-1}$ implying $t_k=(k+1)t_1$.

If $t_k = (k+1)t_1$, then $n=t_{k-1}+t_k = 2kt_1$ is even. Therefore if $n$ is odd, then $t_k = kt_1$.
\end{proof}

We have completely characterized claw-free Toeplitz graphs $T_n\langle t_1, \ldots, t_k \rangle$ for $n \ge t_k+ t_{k-1}$ until now.
It is natural to ask which Toeplitz graphs are claw-free for $t_k+1 \le n < t_k + t_{k-1}$.
To answer this question, we approach from two directions.
One is to consider a relatively small $k$ and the other is to consider a relatively small $n$.
In the rest of this section, we take the second approach while we concentrate on the first approach in Section~\ref{sec:small}.
We start with $n= t_k+1$.
Each of $T_{t_i+1} \langle t_1, \ldots, t_i\rangle$ is an induced subgraph of $T_{t_k+1} \langle t_1, \ldots, t_k \rangle$ since $t_i+1 \le t_{i+1}$ for $i = 1,2,\ldots, k-1$.
Therefore, if $T_{t_k+1} \langle t_1, \ldots, t_k \rangle$ is claw-free, then $T_{t_i+1} \langle t_1,  \ldots, t_i\rangle$ is claw-free for each $i = 1,2, \ldots, k-1$, and if $T_{t_j+1} \langle t_1, \ldots, t_j\rangle$ is not claw-free for some positive integer $j$, then $T_{t_k+1}\langle t_1,  \ldots, t_k \rangle$ is not claw-free for each integer $k \ge j$.
The following example shows that for the Fibonacci sequence $\{t_i\}_{i=1}^\infty$, even if $T_{t_k+1} \langle t_1, \ldots, t_k \rangle$ is not claw-free for $k \ge 6$, it is claw-free for $k \le 5$.

 \begin{exm}\label{ex1}
 Let  $G_k = T_{F_k+1} \langle F_1,\ldots, F_k \rangle$ for the Fibonacci sequence  $\{F_i\}_{i=1}^\infty$ with $F_1 = 1$, $F_2 = 2$.
 Then $G_5  = T_9\langle 1, 2, 3, 5, 8\rangle$ and $G_6 = T_{14}\langle 1,2,3,5,8,13\rangle$.
It is easy to check that  $(3;1,5,11)$ is a claw of $G_6$. Thus $G_k$ is not claw-free for $k \ge 6$.

Now we show that $G_5$ is claw-free.
To reach a contradiction, suppose that $G_5$ contains a claw $(v_1;v_2,v_3,v_4)$ with $v_2 < v_3 < v_4$.
By $(\star)$, we may assume that $v_1 < v_3 <v_4$.
By the definition of Toeplitz graph, $|v_2-v_1|$, $v_3-v_1$, and $v_4-v_1$ are Fibonacci numbers.

First, we suppose $v_1< v_2$. Then $\{v_2-v_1, v_3-v_1,v_4-v_1\} \subset \{1,2,3,5,8\}$.
Since $(v_i -v_1) - (v_j-v_1) = v_i-v_j$ for $2 \le j < i \le 4$, $(v_1;v_2,v_3,v_4)$ is not a claw if two of $v_2-v_1, v_3-v_1,v_4-v_1$ are consecutive Fibonacci numbers.
 Thus $\{v_2-v_1, v_3-v_1,v_4-v_1\}= \{1,3,8\}$. Then $v_4-v_1=8$, $v_3-v_1 = 3$, and $v_2-v_1=1$. However, $v_3-v_2 = 2$ in this case, so $v_2$ and $v_3$ are adjacent and we reach a contradiction.

Now we suppose $v_1>v_2$. Then $\{v_1-v_2, v_3-v_1,v_4-v_1\} \subset \{1,2,3,5,8\}$.
Note that $v_4-v_2 = (v_4-v_1)+(v_1-v_2)$ and $v_3-v_2 = (v_3-v_1)+(v_1-v_2)$.
Since $G_5$ has $9$ vertices, $v_3-v_2 < v_4-v_2 \le 8$ and so  $\{v_1-v_2, v_3-v_1,v_4-v_1\} \subset \{1,2,3,5\}$.
Since $(v_4-v_1) - (v_3-v_1) = v_4-v_3$ and $v_4-v_3 \notin \{1,2,3,5\}$, $v_4-v_1 = 5$ and $v_3-v_1=1$.
Then, since $(v_4-v_1)+(v_1-v_2) = v_4-v_2 \le 8$, $v_1-v_2 \le 3$.
However, if $v_1-v_2$ is $1$ or $2$, then
$v_3-v_2$ is $2$ or $3$, and if $v_1-v_2 =3$, then $v_4-v_2=8$. Consequently,
 we have reached a contradiction.
Thus $G_5$ is claw-free and so $G_1,\ldots,G_5$ are claw-free.
\end{exm}

In the following, we give a characterization for an infinite sequence $\{t_i\}_{i=1}^\infty$ for which $T_{t_k+1} \langle t_1,  \ldots, t_k \rangle$ is claw-free for any positive integer $k$.
We need the following lemma.

\begin{lemma}\label{thm:tkcond}
Suppose that $T_n \langle t_1, \ldots, t_k \rangle$ is claw-free.
Given an integer $1 \le \ell < k$, if there is some positive integer $\ell \le j < k$ such that $t_{j+1}-t_j > t_\ell$, then $t_i = it_1$ for any $1 \le i \le \ell$.
\end{lemma}
\begin{proof}
The Toeplitz graph $T_{t_{j+1}} \langle t_1, \ldots, t_j \rangle$ is an induced subgraph of $T_n \langle t_1, \ldots, t_k \rangle$, so $T_{t_{j+1}} \langle t_1, \ldots, t_j \rangle$ is also claw-free.
Since $j \ge \ell$, $t_i = it_1$ for any $1 \le i \le \ell$ by Theorem~\ref{thm:tktl}.
\end{proof}

\begin{theorem}
Suppose that $T_{t_n+1} \langle t_1, \ldots, t_n \rangle$ is claw-free for any positive integer $n$.
Then the sequence $\{t_n\}^\infty_{n=1}$ is bounded above by some arithmetic sequence with a common difference $t_\ell$ for some positive integer $\ell$.
\end{theorem}

\begin{proof}
If $\{t_n\}$ is an arithmetic sequence, then it is bounded above by itself.
Now, suppose that $\{t_n\}$ is not arithmetic.
Then there is some positive integer $i$ such that $t_i \neq i t_1$.
Let $\ell$ be the smallest integer such that $t_\ell \neq \ell t_1$.
If $t_{j+1} - t_j > t_\ell$ for some $j \ge \ell$, then $t_\ell = \ell t_1$ by Theorem~\ref{thm:tkcond}, which is a contradiction.
Thus $t_{j+1}- t_j \le t_\ell$ for any $j \ge \ell$.
Therefore $t_n \le (n-\ell+1) t_\ell$ for $n \ge \ell$.
Take an arithmetic sequence $\{a_n\}$ with $a_n = nt_\ell$.
If $ n > \ell$, then $t_n \le (n-\ell+1)t_\ell < n t_\ell = a_n$.
If $ n = \ell$, then $t_\ell < \ell t_\ell =a_\ell$.
If $ n < \ell$, then $t_n = nt_1 < nt_\ell = a_n$.
Hence, $\{t_n\}$ is bounded above by arithmetic sequence $\{a_n\}$.
\end{proof}

\section{Claw-free Toeplitz graphs}
\label{sec:small}

In Section~\ref{sec2}, we suggested two approaches to answer the question which Toeplitz graphs are claw-free for $t_k+1 \le n < t_k + t_{k-1}$. One is to consider a relatively small $k$ and the other is to consider a relatively small $n$.
In this section, we characterize claw-free Toeplitz graphs $T_n\langle t_1,t_2\rangle$ and $T_n\langle t_1,t_2,t_3\rangle$ using the second approach.

We begin with the following lemma, which generalizes Proposition~\ref{lem1}.
Given a Toeplitz graph $G:=T_n \langle t_1,\ldots, t_k \rangle$ and an integer $d \ge 2$,
we denote by $G^d_i$ be the subgraph of $G$ induced by the vertex set \[
\{v \in V(G) \mid v \equiv i \pmod{d}\}\] for each integer $1 \le i \le d$.

 \begin{lemma}[\cite{CFP}]\label{lem:iso}
 Let $G=T_n \langle t_1,\ldots, t_k \rangle$ with $\gcd(t_1,\ldots,t_k)=d$.
  Suppose that $r$ is an integer such that $r \equiv n\pmod{d}$ with $1\le r \le d$.
Then each of $G^d_1, \ldots, G^d_r$ is isomorphic to $T_{\lceil n/d \rceil} \langle t_1/d, \ldots, t_k/d \rangle$ and each of $G^d_{r+1}, \ldots, G^d_d$ is isomorphic to $T_{\lfloor n/d \rfloor} \langle t_1/d, \ldots, t_k/d \rangle$.
Moreover, there is no edge joining a vertex in $G^d_i$ and a vertex in $G^d_j$ if $i \neq j$.
\end{lemma}

A {\it hole} is a chordless cycle of length at least $4$ as an induced
 subgraph.

\begin{lemma}[\cite{CJKKM}] \label{lem4}
 Let $G=T_n\langle t_1, t_2\rangle$ with $n\ge t_1+t_2$. If $t_2 \ne 2t_1$, then $G$ has a hole of length $(t_1+t_2)/ \gcd(t_1,t_2)$.
 \end{lemma}

The following result describes a useful structure of the Toeplitz graph $T_n\langle t_1, t_2\rangle$ with $n=t_1+t_2$.
 \begin{theorem} \label{thm:cycles} Let $G=T_n\langle t_1, t_2\rangle$. If $n=t_1+t_2$ and $d=gcd(t_1, t_2)$ then
$G$ is isomorphic to the disjoint union of $d$ copies of the holes of order $n/d$,
i.e.
 \begin{equation*} \label{cycle}
 G\cong d\,C_{ n/d }.
 \end{equation*}
 \end{theorem}
 \begin{proof} Since $n=t_1+t_2$, $n$ is multiple of $d$.
Then, by Lemma~\ref{lem:iso}, $G\cong d\,T_{ n/d } \langle t_1/d, t_2/d \rangle$.
Now, $T_{ n/d } \langle t_1/d, t_2/d \rangle \cong C_{n/d}$ by Lemma~\ref{lem4} and this completes the proof.\end{proof}

A claw-free Toeplitz graph $T_n \langle t_1,t_2 \rangle$ is characterized as follows.
\begin{theorem}\label{thm0}
 A Toeplitz graph $T_n\langle t_1, t_2\rangle$ is claw-free if and
 only if $n\le t_1+t_2$ or $t_2=2t_1$.
 \end{theorem}

 \begin{proof}
If $n > t_1+t_2$, then $t_2=2t_1$ by Theorem~\ref{thm:tktl} and so the `only if' part is valid.

To show the `if" part, suppose that $n\le t_1+t_2$ or $t_2=2t_1$.  If $n=t_1+t_2$, by
Theorem~\ref{thm:cycles}, $G$ is isomorphic to union of $d$ cycles of the same
size where $d={\rm gcd}(t_1,t_2)$. Consequently, for $n< t_1+t_2$,
the graph $G$ is isomorphic to union of isolated vertices, paths and
probably $a$ cycles, where $0\le a \le d-1$. Thus the degree of any
vertex in $G$ is not greater than $2$, which implies $G$ is
claw-free. If $n > t_1+t_2$ and $t_2=2t_1$, then $G$ is claw-free from Lemma~\ref{lem:iso}. Hence the proof is complete.
  \end{proof}

Next, we will characterize a claw-free $T_n\langle t_1,t_2,t_3 \rangle$.
We begin with the following theorem.

\begin{theorem} A Toeplitz graph $G:=T_{t_3+1}\langle t_1, t_2,t_3 \rangle$ is claw-free if and only if $t_1+t_2=t_3$ or $t_2 = 2t_1$ or $t_3 = 2t_1$.\label{thm:clawt3}
\end{theorem}
\begin{proof} To show the `only if' part, suppose that $G$ is claw-free.
If $t_1 + t_2 < t_3$, we consider $H_1 : = [1+t_1;1,1+2t_1,1+t_1+t_2]$, which is a subgraph of $G$ since $1+t_1+t_2 < 1+t_3$.
Otherwise, we consider $H_2:= [1;1+t_1,1+t_2,1+t_3]$.

If $t_1+t_2 < t_3$, then $1$ and $1+t_1+t_2$ are not adjacent since $t_2 < t_1+t_2 < t_3$.
If  $t_1+t_2 > t_3$, then $1+t_2$ and $1+t_3$ are not adjacent since $t_3-t_2 < t_1$.
Therefore, to prevent $H_1$ (resp. $H_2$) from being a claw, at least one of the following is an edge:
$\{1, 1+2t_1\}$, $\{1+2t_1, 1+t_1+t_2\}$ (resp.  $\{1+t_1, 1+t_2\}$,  $\{1+t_1, 1+t_3\}$).
If $1$ and $1+2t_1$ are adjacent, then $t_2 = 2t_1$ since $t_1 <  2t_1 < t_1+t_2 < t_3$.
If $\{1+2t_1, 1+t_1+t_2\}$ or $\{1+t_1, 1+t_2\}$ is an  edge, then $t_2-t_1 = t_1$, i.e.\ $t_2=2t_1$ since $t_2-t_1 <  t_2$.
If $1+t_1$ and $1+t_3$ are adjacent, then $t_3-t_1 = t_1$, i.e.\ $t_3=2t_1$ since $t_3- t_1 < t_2$.
Thus we have shown the `only if' part.

To show the `if' part, suppose that $t_1+t_2 = t_3$ or $t_2 = 2t_1$ or $t_3 = 2t_1$.
To the contrary, suppose that $G$ has a claw $(v_1;v_2,v_3,v_4)$ with $v_2<v_3<v_4$.
By $(\star)$, we may assume $v_1<v_3<v_4$. Then $\{|v_2-v_1|, v_3-v_1, v_4-v_1\} \subset \{t_1,t_2,t_3\}$.

Suppose that $v_4-v_1 = t_3$.
Then $v_4= t_3+1$ and $v_1 = 1$ and so $v_1<v_2$.
Since $v_2-v_1 < v_3-v_1$, we have $v_2=t_1+1$ and $v_3 = t_2+1$.
Therefore \[
\{t_3-t_2, t_3-t_1, t_2-t_1 \} \cap \{t_1, t_2, t_3 \} \neq \emptyset
\]
since $t_1+t_2 = t_3$ or $t_2 = 2t_1$ or $t_3 = 2t_1$.
Thus we have reached a contradiction to the assumption that $(v_1;v_2,v_3,v_4)$ is a claw in each case.

Now we suppose that $v_4 - v_1 \ne t_3$.
Since $v_3-v_1 < v_4-v_1$, $v_3-v_1 = t_1$ and $v_4-v_1=t_2$. Since $v_2-v_1 < v_3-v_1$, $v_2 < v_1$ and so $v_1-v_2 \in \{t_1,t_2,t_3\}$.
If $t_2 = 2t_1$, then $v_4-v_3 = (v_4-v_1)-(v_3-v_1) = t_1$, which is a contradiction.
Thus $t_1+t_2=t_3$ or $t_3 = 2t_1$.
Note that  \[
v_4-v_2 = (v_4-v_1)-(v_2-v_1)=t_2-(v_2-v_1)\]
and
\[
v_3-v_2 = (v_3-v_1)-(v_2-v_1) = t_1 - (v_2-v_1).\]
If $v_1-v_2 = t_3$, then $v_4-v_2 = (v_4-v_1)+(v_1-v_2) = t_2+t_3 >t_3+1$ and we reach a contradiction.
Therefore $v_1 - v_2 \in \{t_1,t_2\}$.
Thus $\{v_4-v_2,v_3-v_2\}$ is either $\{t_1+t_2,2t_1\}$ or $\{2t_2,t_1+t_2\}$.
If $t_1+t_2 = t_3$, then either $v_4-v_2$ or $v_3-v_2$ equals $t_3$, which is a contradiction.
If $t_3= 2t_1$, then either $2t_1< t_1+t_2 = v_4-v_2 \le t_3 $ or $2t_1<2t_2 = v_4-v_2 \le t_3$, which is impossible.
\end{proof}

\begin{theorem}\label{thm:cond3}
Let $G=T_{n}\langle t_1, t_2,t_3 \rangle$ be a Toeplitz graph not in the form $T_n \langle t_1, 2t_1, 3t_1\rangle$.
Then $G$ is claw-free if and only if $n \le t_2+t_3$ and one of the following holds:
\begin{itemize}
\item[{\rm (i)}] $t_3=t_1+t_2$;
\item[{\rm (ii)}] $t_2 = 2t_1$, and $t_3=4t_1$ or $n \le t_1+t_3$;
\item[{\rm (iii)}] $t_3 = 2t_1$ and $n \le \min\{3t_1,2t_2\}$.
\end{itemize}
 \end{theorem}
\begin{proof} We first show the `only if' part.
Suppose that $G$ is claw-free. Since $G$ is not in the form $T_n \langle t_1, 2t_1, 3t_1\rangle$, \begin{equation}n \le t_2+ t_3 \label{nt2t3}\end{equation} by Theorem~\ref{pro2}.
To the contrary, suppose that none of (i), (ii), and (iii) holds.
Assume that none of  $t_1+t_2=t_3$, $t_2 = 2t_1$, and $t_3 = 2t_1$ holds. Then $T_{t_3+1} \langle t_1,t_2,t_3 \rangle$ has a claw by Theorem~\ref{thm:clawt3}. For each $n \ge t_3+1$, $T_{t_3+1} \langle t_1,t_2,t_3 \rangle$ is an induced subgraph of $T_n \langle t_1,t_2,t_3 \rangle$ and so $T_n \langle t_1,t_2,t_3 \rangle$ has a claw.
Therefore one of $t_1+t_2=t_3$, $t_2 = 2t_1$, and $t_3 = 2t_1$ is true.

Suppose that $t_2=2t_1$, $t_3\ne 4t_1$, and $n >t_1+ t_3$. Then $(1+t_1;1,1+t_1+t_2,1+t_1+t_3)$ is a claw.
To see why, suppose that $1$ and $1+t_1+t_2$ are adjacent. Then $t_1+t_2 = t_3$ and $t_3 = 3t_1$, which contradicts the assumption.
Suppose that $1+t_1+t_2$ and $1+t_1+t_3$ are adjacent. Since $t_3\ne 2t_2$, $t_3 - t_2 = t_1$. Then $t_3 = 3t_1$, which contradicts assumption $G\ne T_n \langle t_1, 2t_1, 3t_1\rangle$.

Suppose that $t_3 =2t_1$ and $n > \min\{3t_1,2t_2\}$.
Suppose that $t_1+t_3 \le 2t_2$. Then $(1+t_1;1,1+t_1+t_2,1+t_1+t_3)$ is a claw.
To see why, suppose that $1$ and $1+t_1+t_2$ are adjacent. Then $t_1+t_2 = t_3 = 2t_1$ and so $t_2=t_1$, which is a contradiction.
In addition, since $t_3-t_2 < t_3-t_1 =t_1$, $1+t_1+t_2$ and $1+t_1+t_3$ are not adjacent.
Suppose that $t_1+t_3 > 2t_2$.
Note that $2t_2 > t_1+t_2 > 2t_1 =t_3$. In addition, if $t_2-t_1 =t_1$, then $t_2=2t_1=t_3$, which is impossible.

Now we show the `if' part. For $i,j \in \{t_3+1,\ldots, 2t_3\}$, $T_i \langle t_1,t_2,t_3 \rangle$ is an induced subgraph of $T_j \langle t_1,t_2,t_3 \rangle$ if $i <j$.
Since $T_n \langle t_1, 2t_1, 4t_1\rangle $ is claw-free for $n \le 6t_1$ by Theorem~\ref{thm:equiv},
it suffices to show that the following graphs are claw-free:\begin{itemize}
\item $G_1:=T_{t_2+t_3}  \langle t_1,t_2,t_3 \rangle$ with $t_1+t_2 =t_3$;
\item $G_2:=T_{t_1+t_3} \langle t_1,t_2,t_3 \rangle$ with $t_2=2t_1$ and $t_3\ne4t_1$;
\item $G_3:= T_{3t_1} \langle t_1,t_2,t_3 \rangle$ with $t_3 = 2t_1$ and $3t_1 \le 2t_2$;
\item $G_4 := T_{2t_2} \langle t_1,t_2,t_3 \rangle$ with $t_3 = 2t_1$ and $3t_1 \ge 2t_2$.
\end{itemize}
We note that if a Toeplitz graph has a claw, then the graph has a claw containing the vertex $1$.
First we suppose that one of $G_1$, $G_2$, $G_3$, $G_4$ has a claw with the center $1$. Then it is $(1;1+t_1,1+t_2,1+t_3)$. However, we may check that $1+t_2$ and $1+t_3$ are adjacent in $G_1$, $1+t_1$ and $1+t_2$ are adjacent in $G_2$, and $1+t_1$ and $1+t_3$ are adjacent in $G_3$ and $G_4$.  Thus none of  $G_1$, $G_2$, $G_3$, $G_4$ has a claw with the center $1$.

Fix $i \in \{1, 2, 3, 4\}$.
We suppose that the graph $G_i$ has a claw $H_i$ with a pendent vertex $1$.
Then the possible centers of $H_i$ are $1+t_1$, $1+t_2$, and $1+t_3$.
Consequently, the possible pendant vertices are
\begin{itemize}
\item $1, 1+2t_1, 1+t_1+t_2, 1+t_1+t_3$ for the center $1+t_1$;
\item $1, 1-t_1+t_2, 1+t_1+t_2, 1+2t_2$ for the center $1+t_2$;
\item $1, 1-t_2+t_3, 1-t_1+t_3, 1+t_1+t_3$ for the center $1+t_3$.
\end{itemize}
For each possible center, the following table shows that $H_i$ cannot be a claw by either listing two pairs of vertices that are adjacent in $G_i$ or a pair of vertices that are adjacent in $G_i$ and a vertex that cannot be in $G_i$.
For example, $1+t_1+t_3$ cannot be in $G_4$ since $1+t_1+t_3 = 1+3t_3/2 \ge 2t_2 +1 $.
The other cases may be checked in a similar way.
\begin{center}\resizebox{\columnwidth}{!}{
\begin{tabular}{c|c|c|c}

& $1+t_1$ & $1+t_2$ & $1+t_3$ \\
 \hline
$H_1$ & \begin{tabular}{@{}c@{}}\{$1$, $1+t_1+t_2$\},\\ \{$1+2t_1$, $1+t_1+t_3$\}\end{tabular} & \begin{tabular}{@{}c@{}}\{$1$, $1+t_1+t_2$\},\\ \{$1-t_1+t_2$,$1+2t_2$\}\end{tabular} & \begin{tabular}{@{}c@{}}\{$1$, $1+t_3-t_1$\}, \\\{$1+t_3-t_2$, $1+t_3+t_1$\}\end{tabular} \\
\hline
$H_2$ & \begin{tabular}{@{}c@{}}
\{$1$, $1+t_1+t_2$\} \\ $1+t_1+t_3$\end{tabular} & \begin{tabular}{@{}c@{}}\{$1$, $1+2t_1$\},\\ \{$1-t_1+t_2$,$1+t_1+t_2$\}\end{tabular} & \begin{tabular}{@{}c@{}}\{$1-t_2+t_3$, $1-t_1+t_3$\}, \\$1+t_3+t_1$\end{tabular} \\
\hline
$H_3$ & \begin{tabular}{@{}c@{}}\{$1$, $1+2t_1$\},\\  $1+t_1+t_3$\end{tabular} & \begin{tabular}{@{}c@{}}\{$1-t_1+t_2$, $1+t_1+t_2$\},\\ $1+2t_2$\end{tabular} & \begin{tabular}{@{}c@{}}\{$1$, $1+t_3-t_1$\}, \\ $1+t_3+t_1$\end{tabular} \\
\hline
$H_4$ & \begin{tabular}{@{}c@{}}\{$1$, $1+2t_1$\},\\  $1+t_1+t_3$\end{tabular} & \begin{tabular}{@{}c@{}}\{$1-t_1+t_2$, $1+t_1+t_2$\},\\ $1+2t_2$\end{tabular} & \begin{tabular}{@{}c@{}}\{$1$, $1+t_3-t_1$\}, \\ $1+t_3+t_1$\end{tabular} \\
\end{tabular}}
\end{center}
\end{proof}

 \section{Line Toeplitz graphs}\label{sec3}

 Recall that the line graph $L(H)$ of a graph $H$ is the graph
whose vertex set is in one-to-one correspondence with the set of
edges of $H$, where two vertices of $L(H)$ are adjacent if and only
if the corresponding edges in
 $H$ have a vertex in common. Equivalently, a graph $G$ is a line graph if it is isomorphic to the line graph $L(H)$ of some graph $H$. In this case, we write $G=L(H)$
 \cite{Har}.

\begin{figure}
    \centering
    \begin{tikzpicture}[scale=0.6]
        \vertex (c) at (0,0) {};
        \foreach \i in {1,2,3}{
            \vertex (\i) at (\i*120-90:1.8cm) {};
        }
        \draw (1) -- (c) -- (2) -- (c) -- (3);

        \foreach \i in {0,1,2,3}{
            \node (placeholder\i) at (\i*90:2cm) {};
            \draw (0,-2.8) node{$G_1$};
        }
    \end{tikzpicture}
    \begin{tikzpicture}[scale=0.6]
        \vertex (r) at (1.8,0) {};
        \vertex (t) at (-0.9,1.57) {};
        \vertex (b) at (-0.9,-1.57) {};
        \vertex (c) at (0,0) {};
        \vertex (l) at (-1.8,0) {};
        \draw (r) -- (t) -- (l) -- (c) -- (b) -- (r);
        \draw (t) -- (c) -- (l) -- (b);

        \foreach \i in {0,1,2,3}{
            \node (placeholder\i) at (\i*90:2cm) {};
            \draw (0,-2.8)  node{$G_2$};
        }
    \end{tikzpicture}
    \begin{tikzpicture}[scale=0.6]
        \vertex (t) at (0,1.5) {};
        \vertex (l) at (-1.73,-1.5) {};
        \vertex (r) at (1.73,-1.5) {};
        \vertex (ct) at (0,0.5) {};
        \vertex (cb) at (0,-0.5) {};
        \draw (l) -- (t) -- (r) -- (ct) -- (l) -- (cb) -- (r) -- (l);
        \draw (t) -- (ct) -- (cb);

        \foreach \i in {0,1,2,3}{
            \node (placeholder\i) at (\i*90:2cm) {};
            \draw (0,-2.8) node{$G_3$};
        }
    \end{tikzpicture}
    \begin{tikzpicture}[scale=0.6]
        \vertex (t) at (0,1.5) {};
        \vertex (b) at (0,-1.5) {};
        \vertex (l) at (-1.8,0) {};
        \vertex (r) at (1.8,0) {};
        \vertex (lc) at (-0.8,0) {};
        \vertex (rc) at (0.8,0) {};
        \draw (r) -- (rc) -- (t) -- (b) -- (rc);
        \draw (l) -- (t) -- (lc) -- (b) -- (l) -- (lc);

        \foreach \i in {0,1,2,3}{
            \node (placeholder\i) at (\i*90:2cm) {};
             \draw (0,-2.8)  node{$G_4$};
        }
    \end{tikzpicture}
    \begin{tikzpicture}[scale=0.6]
        \vertex (t) at (0,1.5) {};
        \vertex (b) at (0,-1.5) {};
        \vertex (l) at (-1.8,0) {};
        \vertex (r) at (1.8,0) {};
        \vertex (lc) at (-0.8,0) {};
        \vertex (rc) at (0.8,0) {};
        \draw (r) -- (rc) -- (t) -- (b) -- (rc);
        \draw (l) -- (t) -- (lc) -- (b) -- (l) -- (lc);
        \draw (t) -- (r) -- (b);

        \foreach \i in {0,1,2,3}{
            \node (placeholder\i) at (\i*90:2cm) {};
             \draw (0,-2.8)  node{$G_5$};
        }
    \end{tikzpicture}\\
    \begin{tikzpicture}[scale=0.6]
        \vertex (rt) at (1.8,1.57) {};
        \vertex (rb) at (1.8,-1.57) {};
        \vertex (t) at (-0.9,1.57) {};
        \vertex (b) at (-0.9,-1.57) {};
        \vertex (c) at (0,0) {};
        \vertex (l) at (-1.8,0) {};
        \draw (rt) -- (t) -- (l) -- (c) -- (b) -- (rb);
        \draw (t) -- (c) -- (l) -- (b);

        \foreach \i in {0,1,2,3}{
            \node (placeholder\i) at (\i*90:2cm) {};
            \draw (0,-2.8)  node{$G_6$};
        }
    \end{tikzpicture}
    \begin{tikzpicture}[scale=0.6]
        \vertex (rt) at (1.8,1.57) {};
        \vertex (rb) at (1.8,-1.57) {};
        \vertex (t) at (-0.9,1.57) {};
        \vertex (b) at (-0.9,-1.57) {};
        \vertex (c) at (0,0) {};
        \vertex (l) at (-1.8,0) {};
        \draw (rt) -- (t) -- (l) -- (c) -- (b) -- (rb) -- (rt);
        \draw (t) -- (c) -- (l) -- (b);

        \foreach \i in {0,1,2,3}{
            \node (placeholder\i) at (\i*90:2cm) {};
             \draw (0,-2.8)  node{$G_7$};
        }
    \end{tikzpicture}
    \begin{tikzpicture}[scale=0.6]
        \vertex (ct) at (0,.9) {};
        \vertex (cb) at (0,-.9) {};
        \vertex (rt) at (1.8,.9) {};
        \vertex (rb) at (1.8,-.9) {};
        \vertex (lt) at (-1.8,.9) {};
        \vertex (lb) at (-1.8,-.9) {};
        \draw (lt) -- (lb) -- (cb) -- (lt) -- (ct) -- (cb) -- (rb) -- (ct) -- (rt) -- (rb);

        \foreach \i in {0,1,2,3}{
            \node (placeholder\i) at (\i*90:2cm) {};
             \draw (0,-2.8)  node{$G_8$};
        }
    \end{tikzpicture}
    \begin{tikzpicture}[scale=0.6]
        \vertex (c) at (0,0) {};
        \foreach \i in {0,1,2,3,4}{
            \vertex (\i) at (\i*360/5 +90:1.8cm) {};
            \draw (\i) -- (c);
        }
        \draw (0) -- (1) -- (2) -- (3) -- (4) -- (0);

        \foreach \i in {0,1,2,3}{
            \node (placeholder\i) at (\i*90:2cm) {};
             \draw (0,-2.8)  node{$G_9$};
        }
    \end{tikzpicture}
    \caption{Nine forbidden graphs of line graph}\label{fig:line_forbidden}
\end{figure}

Beineke~\cite{LB} characterized connected line graphs in terms of nine forbidden subgraphs drawn in Figure~\ref{fig:line_forbidden}. This result will be effectively used in this section.

We first characterize line Toeplitz graphs $T_n\langle t,2t,\ldots,kt\rangle$.

\begin{theorem}\label{thm:linegraph}
A Toeplitz graph $T_n\langle t,2t,\ldots,kt\rangle$ is a line graph if and only if one of the following is true:
\begin{itemize}
\item[(i)] $k =2$ and $2t+1 \le n \le 5t$;
\item[(ii)] $k \ge 3$ and $kt+1 \le n \le (k+1)t$.
\end{itemize}
\end{theorem}
\begin{proof}
We first show the `only if' part.
The lower bounds for two cases immediately follows from the definition of Toeplitz graphs.
To show the contrapositive of the `only if' part,
let $H_1 = T_n \langle t,2t \rangle$ with $n >5t$ and $H_2 \langle t,2t,\ldots,kt\rangle$ with $k \ge 3$ and $n > (k+1)t$.
Then the subgraph of $H_1$ induced by $\{1,1+t,1+2t,1+3t,1+4t,1+5t\}$ is isomorphic to $G_8$ in Figure~\ref{fig:line_forbidden}. Therefore $H_1$ is not a line graph.
Since $k \ge 3$, $1$, $1+t$, $1+2t$, $1+3t$, $1+(k+1)t$ are all distinct vertices of $H_2$ and the subgraph of $H_2$ induced by $\{1,1+t,1+(k+1)t,1+2t,1+3t\}$ is isomorphic to $G_3$ in Figure~\ref{fig:line_forbidden}. Therefore $H_2$ is not a line graph and so the `only if' part is valid.

Now we show the `if' part.
By Proposition~\ref{lem1}, each component of $T_n\langle t,2t,\ldots,kt\rangle$ is isomorphic to $T_{\lfloor n/t\rfloor+1}\langle 1, 2, \ldots, k\rangle$.
We consider the case (i) in which $G=T_n \langle t,2t \rangle$ with $n \le 5t$.
Then each component of $G$ is isomorphic to one of $T_2\langle 1 \rangle$, $T_3 \langle 1, 2 \rangle$, $T_4 \langle 1, 2 \rangle$, and $T_5 \langle 1, 2 \rangle$.
It is easy to check that none of these graphs contains as an induced subgraph $G_i$ for any $1 \le i \le 9$ given in Figure~\ref{fig:line_forbidden}.
If the case (ii) happens, that is, $G=T_n\langle t,2t,\ldots,kt\rangle$ with $k \ge 3$ and $kt+1 \le n \le (k+1)t$. Then each component of $G$ is isomorphic to $T_k \langle 1, 2, \ldots, k-1 \rangle$ or $T_{k+1} \langle 1,2, \ldots, k \rangle$, which is a clique, and so $G$ is a line graph.
\end{proof}

Since Toeplitz graphs $T_n\langle t,2t,\ldots,kt\rangle$  have been characterized as mentioned above, we will now proceed to explore Toeplitz graphs not in the form $T_n\langle t,2t,\ldots,kt\rangle$.
Especially, we investigate Toeplitz graphs which have a $k$ value of $3$ or less, similar to Section~\ref{sec:small}.

Proposition~\ref{lem1} shows that $T_n\langle t_1\rangle$ is a disjoint union of paths and so it is a line graph.
When $k$ is $2$, it is  characterized simply as follows.
\begin{theorem}
 \label{prop:line2}
Let $G=T_n \langle t_1,t_2 \rangle$ be a Toeplitz graph with $t_2 \ne 2t_1$. Then $G$ is a line graph if and only if $n \le t_1+t_2$.
\end{theorem}
\begin{proof}
Since any line graph is claw-free, the `only if' part immediately comes from Theorem~\ref{thm0}.

Now we show the `if' part. Suppose that $n + \ell = t_1+t_2$ for some nonnegative integer $\ell$.
By Theorem~\ref{thm:cycles}, $T_{n+\ell} \langle t_1 + t_2 \rangle$ is a disjoint union of cycles.
Since the graph $G$ is a subgraph of  $T_{n+\ell} \langle t_1 + t_2 \rangle $, it is a disjoint union of paths and cycles.
Since any paths and cycles are line graphs, $G$ is a line graph.
\end{proof}

From now on, we will explore the case when $k$ is $3$ and fully resolve this case as follows.

\begin{theorem}\label{thm:line3}
Let $G=T_n \langle t_1,t_2,t_3 \rangle$ be a Toeplitz graph not in the form $T_n \langle t_1, 2t_1, 3t_1\rangle$. Then $G$ is a line graph if and only if $n \le t_1+ t_3$ and one of the following holds:
\begin{itemize}
\item[{\rm (i)}] $t_2 = 2t_1$, and either  $t_3 =4t_1$, $3t_1 < n \le \min\{2t_3 -3t_1, 5t_1-1\}$ or $n=5t_1$;
\item[{\rm (ii)}] $t_3 = 2t_1$ and $n \le 2t_2$.
\end{itemize}
\end{theorem}

Proposition~\ref{prop:line2} and Theorem~\ref{thm:line3} give an equivalent condition ($k=2$ and $n \le 4t$ or $k=3$ and $n \le 5t$) for
$T_n \langle t,3t \rangle$ and $T_n \langle t,2t,4t \rangle$.
Yet, it can be extended to any integer $k \ge 2$ by the following proposition.

\begin{prop} A Toeplitz graph $T_n \langle t,2t,\ldots, (k-1)t,(k+1)t \rangle$ with $k \ge 4$ is not a line graph.
\end{prop}
\begin{proof}
Since $k \ge 4$ and $n > (k+1)t$, the vertices $1$, $1+t$, $1+2t$, $1+3t$, and $1+kt$ are mutually distinct.
Consider the subgraph $H$ of $T_n \langle t_1,\ldots, (k-t)t_1,(k+1)t_1 \rangle$ induced by $\{1,1+t,1+kt,1+2t,1+3t\}$. We can easily check that every pair of vertices in $H$ except $1$ and $1+kt$ is adjacent and
$H\cong G_3$ in Figure~\ref{fig:line_forbidden}.
\end{proof}

In the remaining of this section, we will focus on proving Theorem~\ref{thm:line3}.

\begin{lemma}
Let $G=T_n \langle t_1,t_2,t_3 \rangle$ be a Toeplitz graph not in the form $T_n \langle t_1, 2t_1, 3t_1\rangle$. If $G$ is a line graph, then
$t_3 \ne t_1+t_2$.\end{lemma}
\begin{proof}
We show the contrapositive of the statement.
Suppose that $t_3=t_1+t_2$. If $t_2 = 3t_1$, then $t_3 =4t_1$ and so the subgraph of $G$ induced by $\{1, 1+t_1,1+3t_1, 1+4t_1, 1+2t_1\}$ is $G_2$ given in Figure~\ref{fig:line_forbidden}, which is a forbidden subgraph for line graphs.

Suppose that $t_2 \ne 3t_1$. We claim that the subgraph $H$ of $G$ induced by $\{1, 1+t_1,1+t_2, 1+t_1+t_2, 1+2t_1, 1+t_2-t_1\}$ is isomorphic to $G_6$ or $G_7$ given in Figure~\ref{fig:line_forbidden}, which is a forbidden subgraph for line graphs.
Since $t_2 \ne 2t_1$, we have\[
\{t_2-t_1, t_2-2t_1,2t_1 \} \cap  \{t_1,t_2,t_3\} = \emptyset.\]
The difference between two nonadjacent vertices in $H$ belongs to $\{t_2-t_1, t_2-2t_1,2t_1,t_2-3t_1\}$. If $t_2-3t_1 \ne t_1$ and $3t_1-t_2 \ne t_2$, then $H$ is isomorphic to $G_6$. Otherwise, $H$ is isomorphic to $G_7$.
 Hence $G$ is not a line graph and this completes the proof.
\end{proof}

\begin{lemma}
\label{lem:line124}
A Toeplitz graph $T_n \langle t,2t,4t \rangle$ is a line graph if and only if $n \le 5t$.
\end{lemma}
\begin{proof}
Suppose that  $n > 5t$.
Then the subgraph of $G$ induced by $\{1,1+t,1+2t,1+4t,1+5t\}$ is $G_2$ given in Figure~\ref{fig:line_forbidden}, which is a forbidden subgraph for line graphs.

Suppose that $n \le 5t$. By the definition of Toeplitz graphs, $4t < n$ and so $n = 4t + r$ for some positive integer $r \le t$. Then $G$ is a disjoint union of $r$ copies of $T_5 \langle 1,2,4 \rangle$ and $t-r$ copies of $T_4 \langle 1,2 \rangle$ by Lemma~\ref{lem:iso}. We can easily check that the line graphs of $H_1$ and $H_2$ given in Figure~\ref{fig:line} are isomorphic to  $T_5 \langle 1,2,4 \rangle$ and $T_4 \langle 1,2 \rangle$, respectively.
\end{proof}

\begin{figure}
\centering
\begin{tikzpicture}[scale=0.7]
        \vertex (ct) at (.5,.9) {};
        \vertex (cb) at (.5,-.9) {};
        \vertex (lt) at (-1.3,.9) {};
        \vertex (lb) at (-1.3,-.9) {};
        \draw (lt) -- (lb) -- (cb) -- (lt) -- (ct) -- (cb) ;
 \draw (-.2,-2)  node{$H_1$};
        \foreach \i in {0,1,2,3}{
            \node (placeholder\i) at (\i*90:2cm) {};
        }
\end{tikzpicture}
\begin{tikzpicture}[scale=0.7]
        \vertex (ct) at (0,1.4) {};
        \vertex (cb) at (0,0) {};
        \vertex (lt) at (-1,-1) {};
        \vertex (lb) at (1,-1) {};
        \draw (ct) -- (cb) -- (lt) -- (lb) -- (cb) ;
 \draw (0,-2)  node{$H_2$};
        \foreach \i in {0,1,2,3}{
            \node (placeholder\i) at (\i*90:2cm) {};
        }
\end{tikzpicture}
\caption{Graphs $H_1$ and $H_2$ whose line graphs are isomorphic to $T_5 \langle 1,2,4 \rangle$ and $T_4 \langle 1,2 \rangle$, respectively.}\label{fig:line}
\end{figure}

A collection $\mathcal{P}$ of cliques of a graph $G$ is an \emph{edge clique partition} if every edge of $G$ belongs to at least one clique in $\mathcal{P}$ and an edge clique partition $\mathcal{P}$ of a graph $G$ is said to be a \emph{Krausz decomposition} of $G$ provided that every vertex in $G$ belongs to at most two cliques in $\mathcal{P}$.

Given a graph $G$, a proper subgraph $H$ of $G$, and a component $X$ of $H$, we say that an edge $e \in E(G)-E(H)$ {\it is incident to $X$} if one of its ends belongs to $X$.

\begin{lemma}\label{lem:ghline}
Let $G = T_n \langle t, 2t, u \rangle$ and $H = T_n \langle t, 2t \rangle$ be Toeplitz graphs with $2t<u$. If $n \le t+u$, then the following are true:
\begin{itemize}
\item[{\it (a)}] For each $1 \le i \le n-u$, the vertex $i$ is the smallest in a component of $H$ and the vertex $i+u$ is the largest vertex in a component of $H$.
\item[{\it (b)}] No two distinct edges  $\{i,i+u\}$ and $\{j, j+u\}$ are incident to the same vertex.
\item[{\it (c)}] If $n \le 5t$, then $G$ is a line graph if and only if for any component of order $4$ in $H$, there is at most one edge in $E(G)-E(H)$ which is incident to it.
\end{itemize}
\end{lemma}

\begin{proof}
The part (a) is true since $n \le t+u$, and (b) immediately follows from (a).
To show (c), we suppose $n \le 5t$.

To show the `only if' part, we suppose that there is a component of size $4$ in $H$ such that two edges $\{i,i+u\}$ and $\{j,j+u\}$ are incident to $H_j^{t}$.
By (a), the subgraph induced by the set $\{i, i+u\} \cup \{j, j+u\}\cup V(H_j^{t})$ is isomorphic to $G_6$ in Figure~\ref{fig:line_forbidden} and so $G$ is not a line graph.

To show the `if' part, suppose that 
\statement{S1}{for any component of size $4$ in $H$, there is at most one edge in $E(G)-E(H)$ which is incident to the component.}
It suffices to show that $G$ has a Krausz decomposition.

Let $a$ be an integer such that $a \equiv n \pmod {t}$ with $1 \le a \le t$.
Since $\gcd(t, 2t) = t$, it follows from Lemma~\ref{lem:iso} that
$H_i^{t} \cong T_{\lceil n/t \rceil} \langle 1, 2 \rangle$ if $i=1, \ldots, a$ and $H_i^{t} \cong T_{\lfloor n/t \rfloor} \langle 1, 2 \rangle$ if $i = a+1, \ldots, t$.
Therefore every component $H_i^{t}$ is isomorphic to $T_3\langle 1, 2 \rangle$ or $T_4 \langle 1, 2 \rangle$ or $T_5 \langle 1,2 \rangle$.
See Figure~\ref{fig:krausz} for an illustration.

We construct a Krausz decomposition $\mathcal{K}$ of $H$ in the following way.
If $H_i^t \cong  T_3 \langle 1,2 \rangle$, then we take $\{v_1,v_2,v_3\}$ as a clique.
If $H_i^t \cong  T_4 \langle 1,2 \rangle$, then we take cliques  \begin{equation}\label{eq:123}
\{w_1,w_2,w_3\},\{w_2,w_4\},\{w_3,w_4\}\end{equation} or \begin{equation}\label{eq:234}\{w_1,w_2\},\{w_1,w_3\},\{w_2,w_3,w_4\}.\end{equation}
If $H_i^t \cong  T_5 \langle 1,2 \rangle$, then we take cliques $\{x_1,x_2,x_3\},\{x_2,x_4\},\{x_3,x_4,x_5\}$.
 Then we can easily check that  $\mathcal{K}$ is a Krausz decomposition of $H$.

\begin{figure}
\centering
\begin{tikzpicture}[scale=0.7]
        \vertex (1) at (0,.9) {};
        \vertex (2) at (-1,-.9) {};
        \vertex (3) at (1,-.9) {};
        \draw (1) -- (2) -- (3) -- (1) ;
 \draw (0,-2)  node{};
 \draw (0,1.2) node{\scriptsize{$v_1$}};
 \draw (1.4,-.9) node{\scriptsize{$v_2$}};
 \draw (-1.4,-.9) node{\scriptsize{$v_3$}};
        \foreach \i in {0,1,2}{
            \node (placeholder\i) at (\i*90:2cm) {};
        }
\end{tikzpicture}
\begin{tikzpicture}[scale=0.7]
        \vertex (ct) at (.5,.9) {};
        \vertex (cb) at (.5,-.9) {};
        \vertex (lt) at (-1.3,.9) {};
        \vertex (lb) at (-1.3,-.9) {};
        \draw (ct) -- (lb) -- (cb) -- (ct);
        \draw (lt) -- (cb);
        \draw (lt) -- (lb);
 \draw (.9,.9) node{\scriptsize{$w_1$}};
 \draw (.9,-.9) node{\scriptsize{$w_2$}};
 \draw (-1.7,.9) node{\scriptsize{$w_4$}};
 \draw (-1.7,-.9) node{\scriptsize{$w_3$}};
        \foreach \i in {0,1,2,3}{
            \node (placeholder\i) at (\i*90:2cm) {};
        }
\end{tikzpicture}
\begin{tikzpicture}[scale=0.7]
        \vertex (1) at (0,1.4) {};
        \vertex (5) at (-1.2,0.6) {};
        \vertex (2) at (1.2, 0.6) {};
        \vertex (4) at (-0.8,-1) {};
        \vertex (3) at (0.8,-1) {};
        \draw (1) -- (2) -- (3) -- (1) ;
        \draw (5) -- (4) -- (3) -- (5);
        \draw (2) -- (4);
 \draw (0.1,1.7) node{\scriptsize{$x_1$}};
 \draw (1.6,0.6) node{\scriptsize{$x_2$}};
 \draw (-1.6,0.6) node{\scriptsize{$x_5$}};
 \draw (1.2,-1) node{\scriptsize{$x_3$}};
 \draw (-1.2,-1) node{\scriptsize{$x_4$}};
        \foreach \i in {0,1,2,3,4}{
            \node (placeholder\i) at (\i*90:2cm) {};
        }
\end{tikzpicture}
\caption{The graphs $H_i^t$ isomorphic to $T_3 \langle 1,2 \rangle$, $T_4 \langle 1,2 \rangle$, and $T_5 \langle 1,2 \rangle$, respectively, where $v_1<v_2<v_3$; $w_1<w_2<w_3<w_4$; $x_1<x_2<x_3<x_4<x_5$.}\label{fig:krausz}
\end{figure}

Now we add the edges in $E(G)-E(H)$ to $\mathcal{K}$ as cliques to obtain $\mathcal{K'}$. Then we can easily see that $\mathcal{K'}$ is an edge clique cover of $G$.

We will modify $\mathcal{K'}$ to obtain a Krausz decomposition of $G$.
Take a vertex $v$ in $G$.
Suppose that $v$ is a middle vertex, that is, $v \in \{v_2,w_2,w_3,x_2,x_3,x_4\}$. Then, by (a), $v$ is contained in at most two cliques in $\mathcal{K'}$.
Now suppose that $v$ is the smallest or the largest, that is, $v \in \{v_1,v_3,w_1,w_4,x_1,x_5\}$.
If $v \in \{v_1,v_3,x_1,x_5\}$, then $v$ belongs to exactly one clique in $\mathcal{K}$ and so $v$ belongs to at most two cliques in  $\mathcal{K'}$ by (b).

Suppose that $v$ is $w_1$ (resp. $w_4$) and belongs to more than two cliques. Then $v$ is incident to an edge in $E(G)-E(H)$ and the cliques given in \eqref{eq:234} (resp. \eqref{eq:123}) belong to $\mathcal{K}$ and so belong to $\mathcal{K'}$. We replace the cliques \eqref{eq:234} with \eqref{eq:123} (resp. \eqref{eq:123} with \eqref{eq:234})
in $\mathcal{K}$. Then $w_1$ (resp. $w_4$) belongs to exactly one part.
Then $v$ belongs to exactly two cliques in $\mathcal{K'}$.
In addition, $w_4$ (resp. $w_1$) is not incident to any edge in $E(G)-E(H)$ by (S1) and so every vertex in $H_i^{t}$ belongs to exactly two cliques in the new $\mathcal{K'}$.
Thus we have obtained a Krausz decomposition of $G$ and this completes the proof.
\end{proof}

\begin{theorem}\label{thm:line12}
Let $G = T_n \langle t, 2t, u\rangle$ with $2t < u$ and $u \notin \{3t, 4t\}$.
Then $G$ is a line graph if and only if $n \le t+u$ and one of the following holds:
\begin{itemize}
\item[(i)] $3t < n < 4t$ and $n \le 2u-3t$;
\item[(ii)] $4t < n < 5t$ and $n \le 2u - 3t$;
\item[(iii)] $n=5t$.
\end{itemize}
\end{theorem}

\begin{proof}
Let $H = T_n \langle t, 2t \rangle$ and $a$ be an integer such that $a \equiv n \pmod {t}$ with $1 \le a \le t$.
Since $\gcd(t, 2t) = t$, it follows from Lemma~\ref{lem:iso} that
$H_i^{t} \cong T_{\lceil n/t \rceil} \langle 1, 2 \rangle$ if $i=1, \ldots, a$ and $H_i^{t} \cong T_{\lfloor n/t \rfloor} \langle 1, 2 \rangle$ if $i = a+1, \ldots, t$.
There are edges $\{1, 1+u\}, \{2,2+u\}, \ldots, \{n-u, n\}$ in $G$ that are not in $H$.

To show the `only if' part, suppose that $G$ is a line graph.
If $n > t+u$, then, by Theorem~\ref{thm:cond3}, $G$ is not claw-free and so $G$ is not a line graph.
Suppose that $n\le t+u$.
Since $u\neq 3t$ and $u \neq 4t$, $\{u-t, u-2t\} \cap \{t, 2t\} = \emptyset$.
If $u< 3t$, then the subgraph induced by  $\{1, 1+t, 1+2t, 1+3t, 1+u, 1+3t-u\}$ is isomorphic to $G_6$ given in Figure~\ref{fig:line_forbidden} and so $G$ is not a line graph.

Now suppose that $u \ge 3t$ and $G$ is a line graph.
Then there exists a Krausz decomposition $\mathcal{P}$ of $G$.
Let $r = n-u$.
Then $1 \le r \le t$.

Take a vertex $1 \le v \le r$.
Then $N(v) = \{v+t, v+2t, v+u\}$, which induces the subgraph $P_2 \cup I_1$.
Therefore $\{v, v+t, v+2t\}$ and $\{v, v+u\}$ form the maximal cliques containing $v$.
Thus
\begin{equation}\label{eq:partition1}
\{\{v, v+t, v+2t\}, \{v, v+u\}\} \subseteq \mathcal{P}.
\end{equation}
Since $1 \le v \le r$, $N(v+t) = \{v, v+2t, v+3t\}$.
Since $\{v, v+t, v+2t\} \in \mathcal{P}$ and $\mathcal{P}$ is Krausz decomposition,
\begin{equation}\label{eq:partition2}
\{v+t, v+3t\} \in \mathcal{P}.
\end{equation}

Take a vertex $u+1 \le w \le n$.
Then $N(w) = \{w-u, w-2t, w-t\}$, which induces the subgraph $P_2 \cup I_1$.
Therefore
\begin{equation}\label{eq:partition3}
\{\{w, w-t, w-2t\}, \{w, w-u\}\} \subseteq \mathcal{P}.
\end{equation}

If $r+3t > u$, then $\{r+t, r+2t, r+3t\} \in \mathcal{P}$ by \eqref{eq:partition3} and we reach a contradiction since $\{r+t, r+3t\} \in \mathcal{P}$ by \eqref{eq:partition2}.
Therefore
\begin{equation}\label{ineq:1}
u \ge r+3t.
\end{equation}

Suppose that $n > 5t$.
Then $N(1+3t) = \{1+t, 1+2t, 1+4t, 1+5t\}$.
By \eqref{eq:partition2}, $\{1+t, 1+3t\} \in \mathcal{P}$.
Since $\mathcal{P}$ is a Krausz decomposition, $\{1+2t, 1+3t, 1+4t, 1+5t\} \in \mathcal{P}$ and so must be a clique, which is a contradiction.
Therefore
\begin{equation}\label{ineq:2}
n \le 5t.
\end{equation}
By \eqref{ineq:1} and \eqref{ineq:2}, $3 < n/t \le 5$.

Let $3 < n/t \le 4$.
Then for each $i=1, \ldots, a$ there is a bijection $\psi_i:\{i, i+t, i+2t, i+3t\}\rightarrow\{1, 2, 3, 4\}$ given by $v \mapsto \lceil v/t \rceil$ such that
\begin{equation}\label{eq:iso2}
H_i^{t} \stackrel{\psi_i}{\cong}
T_4 \langle 1,2 \rangle,
\end{equation}
and for each $i= a+1, \ldots, t$ there is a bijection $\psi_i:\{i, i+t, i+2t\} \rightarrow \{1, 2, 3\}$ given by $v \mapsto \lceil v/t \rceil$ such that
\begin{equation}\label{eq:iso2-2}
H_i^{t} \stackrel{\psi_i}{\cong}
T_3 \langle 1,2 \rangle.
\end{equation}

Now let $4 < n/t \le 5$.
Then for each $i=1, \ldots, a$ there is a bijection $\phi_i:\{i, i+t,\ldots , i+4t\}\rightarrow\{1, 2, 3, 4, 5\}$ given by $v \mapsto \lceil v/t \rceil$ such that
\begin{equation}\label{eq:iso1}
H_i^{t} \stackrel{\phi_i}{\cong} T_5 \langle 1,2 \rangle,
\end{equation}
and for each $i=a+1, \ldots, t$ there is a bijection $\phi_i: \{i, i+t, i+2t, i+3t\} \rightarrow \{1, 2, 3, 4\}$ given by $v \mapsto \lceil v/t \rceil$ such that
\begin{equation}\label{eq:iso1-2}
H_i^{t} \stackrel{\phi_i}{\cong} T_4 \langle 1,2 \rangle.
\end{equation}

Since $i+u$ is the largest vertex of some component of $H$, $\psi_j$ and $\phi_j$ map $i+u$ to $3, 4$, or $5$ by Lemma~\ref{lem:ghline}(a).
However, by \eqref{ineq:1}, $i+u\ge u \ge r+3t > 3t$.
In addition, $\psi_j$ and $\phi_j$ map $v$ to $\lceil v/t \rceil$.
Therefore $i+u$ cannot be mapped to $3$.

Furthermore,
\begin{equation}\label{ineq:3}
u \equiv a-r \pmod{t}
\end{equation}
since $n=u+r$ and $n\equiv a \pmod {t}$.

{\it Case 1.} $n< 4t$.
Then $H_1^{t}, \ldots, H_a^{t}$ are the components of $H$ with size $4$.
Suppose $a < r$.
Then there is an edge $\{r-a,r-a+u\}$.
By \eqref{ineq:3}, $r-a+u \equiv t \pmod{t}$.
By \eqref{ineq:1}, $3t < u$.
Then, since $u < n < 4t$, $3t < n < 4t$ and so $a \neq t$.
Therefore $H_{t}^{t} \cong T_3 \langle 1, 2 \rangle$ by \eqref{eq:iso2-2}.
However, we note that $\psi_i$ maps $r-a+u$ to either $4$ or $5$, which is a contradiction.
Thus
\begin{equation}\label{ineq:5}
a \ge r.
\end{equation}
To the contrary, suppose $a< 2r$.
Then $2a-2r+1+3t \le a+3t = n$.
By \eqref{ineq:1}, $u = kt + a-r$ for some integer $k$.
By \eqref{ineq:3} and \eqref{ineq:5}, $u \ge 3t$ and $a-r \ge 0$.
Therefore $k<4$.
Since $a \le t$, $u < (k+1)t$ and so $k=3$ by \eqref{ineq:3}.
Thus $u = 3t + a-r$ and so $2a-2r+1+3t = a-r+1+u$.
Therefore $(a-r+1)+u$ and $1$ belong to the same component $H_1^{t}$ of order $4$.
Thus, by Lemma~\ref{lem:ghline}(c), $G$ is not a line graph and we reach a contradiction.

{\it Case 2.} $4t < n < 5t$ and $2r \ge a+t+1$.
Then $H_{a+1}^{t}, \ldots, H_{t}^{t}$ are the components of size $4$.
Moreover, $2r \ge a+1+t \ge 2(a+1)$ and so
\begin{equation}\label{ineq:4}
r \ge a+1.
\end{equation}
By \eqref{ineq:3}, $u = kt + a-r$ for some integer $k$.
By \eqref{ineq:1} and \eqref{ineq:4}, \[3t+ a < 3t+r \le u = kt+a-r.\]
Therefore $k > 3$.
If $k \ge 5$, then $n= r+ u \ge r+ (5t + a-r) = 5t +a$, which is a contradiction.
Thus $k = 4$ and so \[u = 4t + a-r.\]
Now, \[(2r-a-t)+u = (2r-a-t)+(4t+a-r)=r+3t \in V(H_r^{t}).\]
Thus $H_r^{t}$ of order $4$ contains the smaller end $r$ of the edge $\{r,r+u\}$ and the larger end $2r-a-t+u$ of the edge $\{2r-a-t,2r-a-t+u\}$.
Therefore, by Lemma~\ref{lem:ghline}(c), $G$ is not a line graph and we reach a contradiction.

Now we show the `if' part.

{\it Case 1.} $3t < n < 4t$ and $a \ge 2r$.
Then $r < a-r+1$.
Recall that $E(G)-E(H) = \{ 1(1+u), 2(2+u), \ldots, r(r+u)\}$ and $V(H_i^{t}) = \{i, i+t, i+2t, i+3t\}$ for $1 \le i \le a$.
For $1 \le i \le a-r$, \[i+3t_1 \le 3t_1+a-r = t_3\] and so $H_i^{t_1}$ does not have the larger end of any edge in $E(G)-E(H)$.
For $a-r+1 \le i \le a$, $i > r$ since $a-r+1 > r$.
Thus $H_i^{t_1}$ does not have the smaller end of any edge in $E(G)-E(H)$.
Therefore it cannot happen that a component of size $4$ has the smaller end of an edge in $E(G)-E(H)$ and the larger end of an edge in $E(G)-E(H)$.
Hence $G$ is a line graph  by Lemma~\ref{lem:ghline}.

{\it Case 2.}  $2r < a+t_1+1$.
By \eqref{ineq:3}, \[t_3 = kt_1+a-r\] for some integer $k$.
Suppose $r \le a$.
Since $kt_1 \le t_3 < n < 5t_1$, $k<5$.
Moreover, $4t_1 < n=t_3+r = kt_1 + a$.
Therefore $k= 4$ and so $t_3 \ge 4t_1$.
Then, for any $1 \le i \le r$, the larger end $i+t_3$ of the edge $\{i,i+t_3\}$ in $E(G)-E(H)$ is greater than $4t_1$.
Therefore it cannot happen that a component of size $4$ has the larger end of an edge in $E(G)-E(H)$.
Hence $G$ is a line graph  by Lemma~\ref{lem:ghline}.

Suppose $r>a$.
Then, by \eqref{ineq:1}, \[3t_1+ a < 3t_1+r \le t_3 = kt_1+a-r.\]
Therefore $k > 3$.
If $k \ge 5$, then $n= r+ t_3 \ge r+ (5t_1 + a-r) = 5t_1 +a$, which is a contradiction.
Thus $k = 4$ and so $t_3 = 4t_1 + a-r$.

For $a+1 \le i \le a+t_1-r$, $i+3t_1 \le 4t_1+a-r = t_3$ and so $H_i^{t_1}$ does not have the larger end of any edge in $E(G)-E(H)$.
For $a+t_1-r+1 \le i \le t_1$, $i > r$ since $a+t_1-r+1 > r$.
Thus $H_i^{t_1}$ does not have the smaller end of any edge in $E(G)-E(H)$.
Therefore it cannot happen that a component of size $4$ has the smaller end of an edge in $E(G)-E(H)$ and the larger end of an edge in $E(G)-E(H)$.
Hence $G$ is a line graph  by Lemma~\ref{lem:ghline}.

{\it Case 3.} $n= 5t_1$.
We note that every component has size $5$ and so, by Lemma~\ref{lem:ghline}(c), $G$ is a line graph .
\end{proof}

\begin{proof}[Proof of Theorem~\ref{thm:line3}]

By Theorem~\ref{thm:cond3}, $G$ is claw-free if and only if one of the following holds:
\begin{itemize}
\item $t_2 = 2t_1$ and $t_3 = 4t_1$ or $n \le t_1+t_3$;
\item $t_3 = 2t_1$ and $n \le \min \{3t_1, 2t_2 \}$.
\end{itemize}

If $t_2 = 2t_1$ and $t_3 = 4t_1$, then $G$ is a line graph if and only if $n \le 5t_1$ by Lemma~\ref{lem:line124}.

If $t_2= 2t_1$, $t_3 \ne 4t_1$, and $n \le t_1+t_3$, then $G$ is a line graph if and only if
\[
\text{either }3t_1 < n < 5t_1 \text{ and } n \le 2t_3-3t_1 \quad \text{or} \quad n = 5t_1
\]
by Theorem~\ref{thm:line12}.

By the above observation, it remains to show that (iii) is a sufficient condition for $G$ being a line graph.
To show it by contradiction, suppose that $t_3=2t_1$, $n \le \min\{3t_1, 2t_2\}$, and a subset $\{v_1,v_2,v_3,v_4\}$ of $V(G)$ induces a diamond $K_4-v_1v_4$.
By the definition of Toeplitz graphs, the smallest vertex in $\{v_1,v_2,v_3,v_4\}$ is $1$.
By symmetry, we may assume either $v_1=1$ or $v_2=1$.
We consider the case $v_1=1$. Since $\{v_1,v_2,v_3\}$ is a triangle, $\{v_2,v_3\} = \{1+t_1,1+2t_1\}$.
Since $n \le 3t_1$ and $n\le 2t_2$, $N(1+t_1) = \{1,1+2t_1,1+t_1+t_2\}$ and $N(1+2t_1) = \{1,1+t_1,1+2t_1-t_2\}$.
Then $v_4 = 1+t_1+t_2 = 1+2t_1-t_2$ or $t_1=2t_2$, which is impossible.
Now we consider the case $v_2=1$.
Then $N(v_2) = \{1+t_1,1+t_2,1+2t_1\}$ and both of $\{1+t_1, 1+t_2\}$, $\{1+t_2,1+2t_1\}$ are not edge. Therefore we reach a contradiction.
Therefore $G$ is claw-free and diamond-free.
Hence $G$ is a line graph.
\end{proof}

\section{Closing remarks}
Let $G$ be a Toeplitz graph $T_n\langle t_1, \ldots, t_k\rangle$.
In this paper, we gave a necessary and sufficient condition for $G$ being line graph for $k \le 3$. In addition, we showed that if $n \ge t_k+t_{k-1}$ and $G$ is a line graph, then $t_i=it_1$ for each $i \in [k]$.
 We went further to prove that if $k \ge 4$ and $t_i=it_1$ for each $i \in [k]$, then $G$ is a line graph if and only if $kt_1+1 \le n \le (k+1)t_1$.
In summary, we obtained necessary and sufficient conditions for $G$ being line graph when $k \le 3$ or $n \ge t_k+t_{k-1}$. Based on our results, we propose the following problem.
 \begin{prb}
 Characterize all line Toeplitz graphs $G=T_n\langle t_1, \ldots, t_k\rangle$ with $k\ge 4$ and order $1+t_k\le n < t_k+t_{k-1}$.
 \end{prb}

\section{Statements and Declarations}
\subsection{Funding}
This work was supported by the National Research Foundation of Korea (NRF) grant funded by the Korea government (MSIP)
(2016R1A5A1008055, NRF-2019R1A2C1007518, 2021R1C1C2014187 and NRF-2022R1A2C1009648).
\subsection{Competing interest}
The authors have no relevant financial or non-financial interests to disclose.
\subsection{Authorship Principles}
All authors contributed to the study conception and design.

\bibliographystyle{plain}

 \end{document}